\documentclass[11pt]{amsart}
\usepackage{amsbsy,amssymb,amscd,amsfonts,latexsym,amstext,delarray, amsmath,graphicx,color,caption,blkarray,mathtools,comment}
\usepackage{tikz}
 
\usetikzlibrary{cd}
\usetikzlibrary{matrix,arrows}
\usepackage{graphicx}
\usepackage{hyperref}
\input xy

\xyoption{all}
\usepackage{euscript}
\usepackage{multirow}
\usepackage[utf8]{inputenc}
\usepackage{etex, pictexwd,dcpic}
\usepackage[normalem]{ulem}

\DeclareMathOperator{\wts}{wts}

\newcounter{thecounter}
\numberwithin{thecounter}{section}
\newtheorem{lemma}[thecounter]{Lemma}
\newtheorem{proposition}[thecounter]{Proposition}
\newtheorem{theorem}[thecounter]{Theorem}

\newtheorem{corollary}[thecounter]{Corollary}

\theoremstyle{definition}

\DeclareMathOperator{\SL}{SL}
\DeclareMathOperator{\GL}{GL}

\DeclareMathOperator{\Sp}{Sp}

\DeclareMathOperator{\Aut}{Aut}

\DeclareMathOperator{\End}{End}

\DeclareMathOperator{\htt}{ht}
\DeclareMathOperator{\re}{{re}}

\DeclareMathOperator{\dep}{{depth}}

\renewcommand{\a}{\alpha}
\renewcommand{\b}{\beta}

\newcommand{\Z}{{\mathbb Z}}
\newcommand{\R}{{\mathbb R}}
\newcommand{\C}{{\mathbb C}}
\newcommand{\N}{{\mathbb N}}
\newcommand{\Q}{{\mathbb Q}}

\newcommand{\E}{{\mathbb E}}
\newcommand{\F}{{\mathbb F}}
\newcommand{\K}{{\mathbb K}}

\newcommand{\fh}{{\mathfrak{h}}}
\newcommand{\fg}{{\mathfrak{g}}}
\newcommand{\fn}{{\mathfrak{n}}}
\newcommand{\calU}{{\mathcal{U}}}

    \advance \topmargin by -\headheight
    \advance \topmargin by -\headsep

    \evensidemargin \oddsidemargin
\begin{document}
\title{Chevalley groups over $\Z$:\\
A representation--theoretic approach}

\author{Abid Ali, Lisa Carbone and Scott H.\ Murray}

\date{\today}

\excludecomment{longver}\includecomment{shortver}


\begin{abstract} 
Let $G(\Q)$ be a  simply connected Chevalley group over $\Q$ corresponding to a simple Lie algebra $\mathfrak g$ over $\C$. 
Let $V$ be a finite dimensional faithful highest weight $\mathfrak g$-module and let $V_\Z$ be a Chevalley $\Z$-form of~$V$. 
Let  $\Gamma(\Z)$ be the subgroup of $G(\Q)$ that preserves $V_{\Z}$   and let $G(\Z)$ be the group of $\Z$-points of $G(\Q)$. 
Then $G(\Q)$ is \emph{integral} if  $G(\Z)=\Gamma(\Z)$. Chevalley's original work constructs a scheme-theoretic integral form of $G(\Q)$ which equals $\Gamma(\Z)$. Here we give a representation-theoretic proof of integrality of $G(\Q)$ using only the action of  $G(\Q)$ on $V$, rather than the language of group schemes. We discuss the challenges and open problems that arise in trying to extend this to a proof of integrality for Kac--Moody groups over $\Q$.
\end{abstract}

\maketitle

\begin{center}
{\it Dedication}
\end{center}
\begin{center}
To the memory of Nicolai Vavilov, whose beautiful papers are a great inspiration.
\end{center}

\begin{longver}
\tableofcontents
\end{longver}

\section{Introduction}\label{S-intro}
In his seminal work \cite{Chevalley,Chevalley2},  Chevalley gave   a unified construction of semisimple linear algebraic groups over arbitrary fields by constructing (what was his notion of) an affine group scheme $G_\Z$ over $\Z$, corresponding  to  a semisimple Lie algebra $\frak g$ over $\C$ with  root system $\Delta$.
Chevalley's  affine group scheme  was later replaced by the modern definition due to   Demazure and Grothendieck \cite{SGA3}.

Let $V$ be a faithful finite dimensional highest weight  $\mathfrak{g}$-module with representation $\rho:\mathfrak{g}\to {\rm End}(V)$ and
let $L$ be the lattice generated by $\wts(V)$, the weights of $V$.
Then $L$  satisfies $Q\subseteq L\subseteq P$, where $Q$ is the root lattice and $P$ is the weight lattice of $\frak g$. Chevalley
\cite{Chevalley}
associated a group scheme $G_\Z$ to $(\Delta,L)$ and an integral form $V_\Z$  of $V$, which has the property that $V_\Z$ is stable under a Cartier--Kostant integral form $\mathcal{U}_\Z$ of the universal enveloping algebra $\mathcal{U}$ of $\mathfrak g$. Such a lattice in $V$ is called an {\it admissible lattice}. 
Let $G_\Z(\Q)$ be the group of $\Q$-rational points of $G_\Z$ and let
$$\Gamma(\Z) = \{g \in G_\Z(\Q)\mid g\cdot V_{\Z} = V_{\Z}\}$$ 
be the subgroup of $G_\Z(\Q)$ that preserves $V_{\Z}$. Chevalley \cite{Chevalley} showed that, up to canonical isomorphism, the group $\Gamma(\Z)$  depends only on ($\Delta$, $L$) and not on the choice of admissible lattice~$V_{\Z}$.

In this work, we are interested in those properties of the groups $G_\Z(\Q)$, $G_\Z(\Z)$ and $\Gamma(\Z)$ that are determined by their actions on $V$ and $V_\Z$. 

Let $\mathfrak h$ be a Cartan subalgebra of $\mathfrak g$ with dual $\mathfrak h^*$. Let $\Delta\subseteq\fh^*$ be the root system of $\fg$ with  set of positive roots (respectively negative roots) $\Delta_{+}$ (respectively $\Delta_{-}$) and simple roots $\alpha_1,\dots,\a_\ell$.  
An essential component of Chevalley's work (Theorem~\ref{Cbasis}) was to prove the existence of a basis   
$\mathcal{B}=\{x_{\a},h_{i}\mid \a\in\Delta,\ h_{i}\in \fh,\ i\in I \}$
for $\mathfrak g$ such that the structure constants with respect to $\mathcal{B}$ are integers.

One can verify (see Section~\ref{universal}) that for each root $\a\in\Delta$, the endomorphism $\rho(x_{\a})^n$ is zero for $n$ sufficiently large. It follows that for $t\in\Q$
\begin{eqnarray}\label{chidef}
   \chi_{\a}(t):=\exp(t\cdot \rho(x_{\a}))=I+t \rho(x_{\a})+\frac{t^2}{2}(\rho(x_{\a}))^2+\dots  
\end{eqnarray}
is a finite sum and hence a well defined automorphism of $V$.
Using the weight space decomposition $V=\bigoplus_{\mu\in\wts(V)} V_\mu$, we can also define automorphisms 
$h_i(t)$ for all $i\in I$ and $t\in \Q^{\times}$ by
\begin{eqnarray}\label{torele}
    h_i(t)\cdot v := t^{\langle\mu,\a_i^{\vee}\rangle }v
\end{eqnarray}
for all $v\in V_\mu$.

Following \cite{Carter}, \cite{CMT} and \cite{Steinberg}, we define Chevalley groups in the following way.
Given an integral domain $\F$, the Chevalley group $G(\F)$ is
$$G(\F) = \left\langle \chi_\a(u), h_i(t) \mid \a\in\Delta, u\in\F, i\in I, t\in \F^\times\right\rangle.$$
 Defining relations for $G(\F)$ were given by Steinberg \cite{Steinberg}. If $L=Q$, then $G(\F)$ is called the {\it adjoint Chevalley group} and if $L=P$,  $G(\F)$ is called the {\it simply connected Chevalley group}.
 

Taking $\F=\Q$, we say that a subgroup of $G(\Q)$ is {\it generated over $\Z$} if it has a generating set of the form $\{\chi_{\a}(u), h_i(t)\mid \a\in S,\ u\in\Z,\  i\in I, \ t\in \Z^\times\}$
 for some subset $S\subseteq \Delta$.

 We let $G(\Z)$ denote the subgroup of $G(\Q)$ generated over $\Z$ by the elements 
 $$\{\chi_{\a}(u), \  h_i(t)\mid \a\in\Delta,\ u\in\Z,  i\in I, t\in \Z^\times\}.$$
 Then 
$$G(\Z)=\langle\chi_{\a}(u),\   h_i(-1)\mid \a\in\Delta,\ u\in\Z,\   i\in I \rangle.$$

For a fixed $V$ and $V_\Z$, we say that $G(\Q)$ is {\it integral} if  $$G(\Z)=\Gamma(\Z).$$
 
Chevalley's original work constructs a scheme-theoretic integral form $G_\Z(\Z)$ of $G(\Q)$ which equals $\Gamma(\Z)$.

 Our approach differs from that of Chevalley in that we use the definition of Chevalley groups as in \cite{Carter}, \cite{CMT}, \cite{Steinberg} described above. In our representation-theoretic setting,  the equality $G(\Z)=\Gamma(\Z)$ must be proven explicitly.

In certain cases, it is straightforward to verify this equality using properties of the Chevalley group and the underlying representation. See for example Section ~\ref{integrality} where there is a simple proof that $\SL_2(\Z)$ is the stabilizer of the standard lattice $\Z\oplus\Z$ in $\Q\oplus\Q$. See also \cite{Soule}, where Soul\'e showed that the Chevalley group $\E_7(\Z)$ (non-compact form) coincides with  $\E_7(\R)\cap \Sp_{56}(\Z)$ where $\Sp_{56}(\Z)$ is the stabilizer of the standard lattice, with the canonical symplectic form, in the fundamental representation of the Lie algebra $\frak e_7$ of dimension 56.

Here we give a proof that $G(\Z)=\Gamma(\Z)$ for  simply connected Chevalley groups $G(\Q)$ using only properties of the action of  the group $G(\Q)$ on the representation $V$. We first prove that (Theorem~\ref{stabilizer0})
$$U(\Q)\cap \Gamma(\Z)=U(\Z)$$
where 
\begin{align*}
U(\Q)&:=\langle \chi_{\alpha}(t)\mid \alpha\in\Delta^+,\ t\in\Q\rangle\\
\intertext{is the positive unipotent subgroup of $G(\Q)$ associated to a root system $\Delta$ and simple Lie algebra $\frak g$, and }
U(\Z)&:=\langle \chi_{\alpha}(t)\mid \alpha\in \Delta^+,\ t\in\Z\rangle
\end{align*}
is the subgroup of $U(\Q)$ generated over $\Z$.

We then use  a decomposition of $G(\Q)$ as $G({\mathbb{Q}})=G({\mathbb{Z}})B({\mathbb{Q}})$ (\cite[Theorem~18]{Steinberg} and Theorem~\ref{Bruhat}) to show that $G(\Z)=\Gamma(\Z)$ (Theorem~\ref{stabilizer1}).
Here $B(\Q)$ denotes the Borel subgroup of $G(\Q)$. The decomposition $G({\mathbb{Q}})=G({\mathbb{Z}})B({\mathbb{Q}})$ is an analog of the Iwasawa decomposition of $G(\Q_p)$, where $G({\mathbb{Z}})$ plays the role of the maximal compact subgroup  $G(\Z_p)$. 

Part of the motivation for this work is to extend the proof of integrality given here to integrality of representation theoretic Kac--Moody groups $G_V(\Q)$ over $\Q$. Here, $G_V(\Q)$ is  constructed with respect to an integrable highest weight module $V=V^{\lambda}$ for the underlying symmetrizable Kac--Moody algebra $\mathfrak{g}$, with dominant integral highest weight~$\lambda$.

Over a general commutative ring $R$, there is still no widely agreed upon definition of a Kac--Moody group $G(R)$, except in the affine case (\cite{Allcock}, \cite{Garland}). 
A proof of integrality for $G_V(\Q)$ would give a unique definition of representation theoretic Kac--Moody groups over $\Z$.
There are certain difficulties and open questions  in answering this question for Kac--Moody groups. The  issues that arise are discussed in Section~\ref{KMcase}.

In \cite{ACLM}, we addressed the question of integrality of certain subgroups of a Kac--Moody group $G_V(\Q)$ associated to a symmetrizable Kac--Moody algebra $\mathfrak{g}=\mathfrak g(A)$ with generalized Cartan matrix $A$, using only the definition of $G_V(\Q)$ acting on an integrable highest weight module $V$. In the Kac--Moody case, the construction of the group $G_V(\Q)$ depends on the choice of integrable highest weight representation $V=V^\lambda$ corresponding to a dominant integral weight $\lambda$. 
In this setting, we called a subgroup $M$ of $G_V(\Q)$  {\it integral} if $M\cap G_V(\Z)= M\cap \Gamma(\Z)$, where $G_V(\Z)$ is the subgroup of $G_V(\Q)$ generated over $\Z$.  This is equivalent to showing that for all $g\in M$, $g\cdot V_\Z \subseteq V_\Z$ implies that $g\in G_V(\Z)$. Let $U_V(\Q)$ be the positive unipotent subgroup of $G_V(\Q)$. In \cite{ACLM}, we proved that, for all  $w\in W$, the \emph{inversion subgroup}
 $$U_{(w)}(\Q)=\langle U_{\beta}(\Q)\mid \beta\in \Phi_{(w)}\rangle,$$ 
 is integral, where $
 \Phi_{(w)}$ is a subset of the set of positive  real roots $\Delta^{\re}_+$,  defined as follows
$$\Phi_{(w)}=\{\beta\in \Delta^{\re}_+\mid w^{-1}\beta\in \Delta^{-}\},$$ where $U_\beta(\Q)$ is the root group corresponding to $\beta\in \Delta^{\re}_+$. We briefly discuss this approach in Section~\ref{inversion}.

Unfortunately, we don't currently know how to extend this to a proof of integrality for unipotent subgroups of Kac--Moody groups in general. 
However, this method gives a proof of integrality of $U_V(\Q)$ when the generalized Cartan matrix $A$ has {\it finite type}. That is, when $A$ is a Cartan matrix and $G_V(\Q)$ coincides with the simply connected semisimple algebraic group associated to $\mathfrak{g}=\mathfrak g(A)$ over $\Q$.  To summarize this alternate approach to proving integrality, we discuss this proof briefly in Section~\ref{inversion}. 
We believe that an analog of Theorem~\ref{stabilizer1} is needed to address the integrality question for Kac--Moody groups.

The authors are grateful to  Shrawan Kumar for his interest in this work and for helpful discussions.

\section{Preliminaries}\label{S-prelim}
Our primary reference for Chevalley groups and  related notions is \cite{Steinberg}. Other useful references are \cite{Bourbaki7-9, Carter,  Chevalley, Humphreys,VP}.

We make frequent use of the concepts of $\Z$-forms and $\R$-forms,  so we recall the definition:
Let $A$ denote an algebra,  an $A$-module, or a vector space, respectively, over $\C$,  and let $R$ be a subring of $\C$.
Then $B\subseteq A$ is an {\it $R$-form} of $A$ if $B$ is an algebra, a $B$-module, an $R$-module, or a vector space  over $R$ respectively, and the natural map $B\otimes_R \C\rightarrow A$ is an isomorphism.

\subsection{Complex semisiple Lie algebras}
In this subsection we review some basic theory of   complex semisimple Lie algebras. 
Let $\mathfrak g$ be a semisimple Lie algebra over $\C$. 
We fix a \emph{Cartan subalgebra}, that is, a nilpotent subalgebra $\fh\subseteq\fg$ with $[\fh,\fh]\subseteq\fh$. 
Since $\fg$ is semisimple, $\fh$ is abelian.
We write $\ell=\dim\fh$ for the \emph{rank} of $\fg$.

For $\alpha\in \mathfrak{h}^*$, we set 
$$\mathfrak{g}_{\alpha} := \{x\in \mathfrak{g}\mid[h,x]=\alpha(h)x,\text{ for all 
$h\in \mathfrak{h}$}\}.$$
We have $\fg_0=\fh$. 
If $\a\ne0$ and $\fg_\a\ne0$, then 
we
call $\alpha$ a \emph{root} and $\fg_\a$ the corresponding \emph{root space}.
Let $\Delta$ denote the set of all roots.  

Let $(\circ,\circ)$ denote the Killing form on  $\mathfrak{g}$. This is a 
nondegenerate symmetric  bilinear form whose restriction to  $\mathfrak{h}$ is also nondegenerate. We also use $(\circ,\circ)$ to denote the nondegenerate  form induced on $\mathfrak{h}^*$.
If we restrict the form on $\fh^\ast$ to the
the $\R$-vector space $E:=\R\Delta$, then $E$
is a Euclidean space and an $\R$-form of $\fh^*$.
For a root $\alpha$, we define a \emph{root reflection}
$$w_\a:\quad E\rightarrow E,\quad  x\mapsto x-\frac{2(\a,x)}{(\a,\a)}\a$$
and we have $w_\a(\Delta)\subseteq \Delta$, for all $\a\in\Delta$.
Hence the set of roots $\Delta$ forms a {\it root system} in $E$. 
Furthermore, $\Delta$ is \emph{crystallographic}, that is,
$\frac{2(\beta,\alpha)}{(\alpha,\alpha)}$ is an integer for every $\alpha,\beta\in \Delta$.
The matrix 
$$A:=\left(\dfrac{2(\alpha_i,\alpha_j)}{(\alpha_j,\alpha_j)} \right)_{i,j\in I}$$
is called the \emph{Cartan matrix} of~$\fg$.

Let $\pi: E\rightarrow \R$ be a fixed map such that $\ker(\pi)$ is a hyperplane that does not intersect $\Delta$. 
We say $\a\in\Delta$ is \emph{positive} (respectively\ \emph{negative}) if $\pi(\a)$ is positive
(respectively\ negative).
The set of positive (respectively\ negative) roots is denoted $\Delta^\pm$.
A positive root $\alpha$ is called a \emph{simple root}  if it cannot be written as the sum of two positive roots.
The simple roots form a basis of $E$ and are denoted  
$\alpha_1,\dots,\alpha_\ell$. 
For convenience we write $I=\{1,\dots,\ell\}$ for the indexing set of the simple roots.
Hence every  $\alpha\in\Delta^+$ has an expression of the form 
$\alpha=\sum_{i\in I} a_i\alpha_i$ where the $a_i\ge0$ are  integers.
The \emph{height} of the root $\a=\sum_{i=1}^{\ell} a_i\alpha_i$ is 
$$
  \htt(\a):= \sum_{i\in I} a_i.  
$$

The reflections $w_\a$, for $\a\in\Delta$, generate a \emph{reflection group} 
$$W=\langle w_\a\mid \a\in\Delta\rangle\subseteq \GL(E),$$
called the \emph{Weyl group} of $\fg$.  
Using the Cartan matrix $A=(A_{ij})_{i,j\in I}$, we define the \emph{Coxeter matrix}
$(c_{ij})_{i,j\in I}$ by
$$
c_{ij} := \begin{cases} 
  1 &\text{if $i=j$,}\\
  2 &\text{if $A_{ij}A_{ji}=0$,}\\
  3 &\text{if $A_{ij}A_{ji}=1$,}\\
  4 &\text{if $A_{ij}A_{ji}=2$,}\\
  6 &\text{if $A_{ij}A_{ji}=3$.}\\
 \end{cases}
$$
If we write $w_i:=w_{\a_i}$, then $W$ is a \emph{Coxeter group} with presentation
$$W=\langle w_1,\dots,w_\ell \mid (w_iw_j)^{c_{ij}}=1 \text{ for $i,j\in I$}
\rangle.$$
We can write $w\in W$ as a word of minimum length in the generators $w_1,\dots,w_\ell$ of the form
$w=w_{j_{1}}w_{j_{2}}\dots w_{j_{m}}$. This is called a \emph{reduced word} for $w$ and
$\ell(w):= m$ is called the \emph{length} of $w$.

\begin{lemma}\label{weroot}
For any root $\alpha\in\Delta$, we have $\alpha=w\alpha_i$ for some $w\in W$ and some simple root $\alpha_i$.
\end{lemma}

The semisimple Lie algebras over $\C$ are classified by their corresponding root systems.
The root systems are classified by Dynkin diagram or, equivalently, by Cartan type ($A_n$, $B_n$ and so on).
This also gives $\fg$ a \emph{root space decomposition}
$$\fg=\fh\oplus\bigoplus_{\a\in\Delta} \fg_\a$$
where each $\fg_\a$ is one dimensional, and a \emph{triangular decomposition} 
$$\mathfrak{g}=\mathfrak{n}^-\oplus\mathfrak{h}\oplus\mathfrak{n}^+,$$
where
$\mathfrak{n}^+ =\bigoplus_{\alpha\in\Delta^+}\mathfrak{g}_{\alpha}$ and $
\mathfrak{n}^- = \bigoplus_{\alpha\in\Delta^-}\mathfrak{g}_{\alpha}$ are nilpotent subalgebras.

\subsection{Root lattice and weight lattice}
Since $\Delta$ is crystallographic, the roots generate a  lattice $$Q:=\Z\Delta=\bigoplus_{i\in I}\Z\a_i\subseteq\mathfrak{h}^{\ast}$$ called the \emph{root lattice} with $WQ=Q$.
 The Lie algebra $\fg$ is a $Q$-graded algebra since $[\fg_\a,\fg_\b] \subseteq \fg_{\a+\b}$, for all $\a,\b\in Q$.

\begin{lemma}[\cite{Steinberg}, Lemma 1]\label{coroot}  For each root $\alpha\in\Delta$, there is a unique
$h_\a\in\fh$ such that
$$(h,h_{\alpha})=\frac{2\a(h)}{(\alpha,\alpha)}$$
for all $h\in\fh$. If we write $h_i:=h_{\a_i}$,
then each $h_{\alpha}$ is a $\Z$-linear combination of the $h_i$.  
\end{lemma}

The \emph{coroot} $h_\a$ is also written $\a^\vee$, and $\Delta^\vee=\{\a^\vee\mid\a\in\Delta\}$.
We take $\langle\circ,\circ\rangle$ to be the natural pairing $\fh^\star\times \fh\to\C$ of $\fh^\star$ with $\fh$, that is, 
$$\langle \varphi,h\rangle := \varphi(h)$$
for all $\varphi\in\fh^\star$ and $h\in\fh$.
For $\a,\beta\in\Delta$, we have 
$$
 \langle \alpha, h_{\beta}\rangle=\dfrac{2(\beta,\alpha)}{(\alpha,\alpha)},   $$ 
and so $\langle\alpha_i,h_j\rangle =A_{ij}$, for all $i,j\in I.$
We denote the set of coroots by $\Delta^\vee$. 
The \emph{coroot lattice} is
$$ Q^\vee = \Z\Delta^\vee = \bigoplus_{i\in I} \Z h_i.$$

The \emph{weight lattice} $P$ is the dual of the coroot lattice $Q^\vee$ under the pairing $\langle\circ,\circ\rangle$. 
That is,
 $$P=(Q^\vee)^*=\{\lambda\in \mathfrak{h}^{\ast}\mid \langle \lambda,h\rangle \in\Z \text{ for all $h\in Q^\vee$}\}.$$
The weight lattice has a basis of \emph{fundamental weights} 
$\omega_1,\dots ,\omega_{\ell}$ such that
$\langle \omega_i,h_j\rangle=\delta_{ij}$,where
$$\delta_{ij}:=
\begin{cases} 1&\mbox{if } i=j\\
0&\mbox{if }i\neq j \end{cases}$$
is the Kronecker delta. 

Similarly the \emph{coweight lattice}
$$P^\vee=Q^*=\{ h\in \mathfrak{h} \mid \langle \lambda,h\rangle \in\Z\text{ for all $\lambda\in Q$}\}$$
has a basis of fundamental coweights 
$\omega^\vee_1,\dots ,\omega^\vee_{\ell}$ such that
$\langle \a_i,\omega^\vee_j\rangle=\delta_{ij}$.
If $\a\in Q$ then~$\a = \sum_{i\in I} \langle\a,\omega^\vee_i\rangle \a_i$,
and if $\lambda\in P$ then
$\lambda=\sum_{i\in I} \langle \lambda,h_i\rangle \omega_i$.

Since $\langle\alpha_j,\alpha_i^{\vee}\rangle=A_{ij}\in\Z$ for $i,j\in I$, we have
$\alpha_i\in P$,
so roots are weights and $Q\leq P$. 
Given a lattice $L$ such that $Q\subseteq L\subseteq P$, we can define a dual lattice $$L^*=\{ h\in \mathfrak{h} \mid \langle h,y\rangle \in\Z\text{ for all $y\in L$}\}.$$
Given lattices $Q\subseteq L\subseteq M\subseteq P$ we have $M^*\subseteq L^*$.
It is straightforward to show that 
$P/Q$ is a finite abelian group of order $\left|\det A\right|$ where $A$ is the Cartan matrix.
See Figure~\ref{F-lattices}.
\begin{figure}[h]
\begin{center}
\begin{tikzcd}
                        & \mathfrak{h}                                   &  &  &  & \mathfrak{h}^\ast  &                                                              \\
\text{coweights}          & P^\vee \arrow[rrrrddd, dotted,leftrightarrow] \arrow[u, hook] &  &  &  & P \arrow[u, hook]  & \text{weights}  \\
                        &                                                &  &  &  & L \arrow[u, hook]  &                                                                \\
                        & L^* \arrow[rrrru, dotted,leftrightarrow] \arrow[uu, hook]  &  &  &  &                    &                                                                \\
\text{coroots} & Q^\vee \arrow[rrrruuu, dotted,leftrightarrow] \arrow[u, hook] &  &  &  & Q \arrow[uu, hook] & \text{roots}                                     
\end{tikzcd}
\end{center}
\caption{Lattices}\label{F-lattices}
\end{figure}
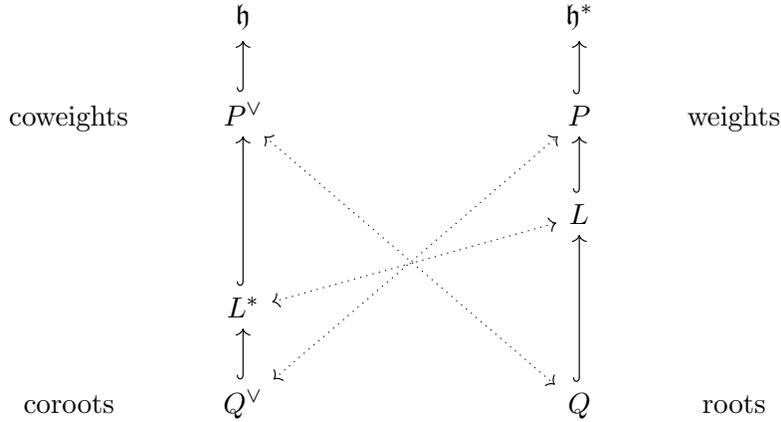

\subsection{Chevalley basis}\label{SS-chevbas}
If $\a$ is a root, then the root space $\mathfrak{g}_{\alpha}$ has dimension one. 
For roots $\a$ and $\b$, 
recall that the constant $p_{\a\b}\in\Z_{\ge0}$ is defined so that
$\b,\b+\a,\dots, \b+p_{\a\b}\a\in\Delta$ but $\b+(p_{\a\b}+1)\b\notin\Delta$. 
Chevalley showed that there is a basis  for $\mathfrak g$ with the following properties.
\begin{theorem} \label{Cbasis} 
Every semisimple Lie algebra $\fg$ over $\C$
has a basis 
$$\mathcal{B}=\{x_{\a},h_{i}\mid \a\in\Delta,\ i\in I \}$$
such that
\begin{align*}
h_{i}&\in \fh,\\ 
[h_{i},x_{\a}]&=\a(h_i)  \\
[x_{\a},x_{-\a}]&=h_{\alpha}\in\bigoplus_{i=1}^{\ell} \Z h_{i},\\
[x_{\a},x_{\beta}]&=\begin{cases} n_{\a\beta} x_{\a+\beta}&\mbox{if }  \a+\beta\in\Delta,\ \a\neq -\beta\\
0& \mbox{otherwise, }  \end{cases}
\end{align*}
where $n_{\a\b}=\pm (p_{\a\b}+1)$,
and such that the linear map 
$$\theta:\quad \frak g\to\frak g, \quad h_i\mapsto - h_i, \quad x_\a\mapsto x_{-\a}$$
is a Lie algebra automorphism
\end{theorem}
Hence  the the structure constants with respect to $\mathcal{B}$ are integers.
Such a basis is called a \emph{Chevalley basis} and  $\theta$ is called the \emph{Chevalley involution}. The structure constants $n_{\a,\beta}\in\C$ depend only on the roots $\a,\b$ and a choice of sign. See \cite{CMT} for more details about which choices of signs are consistent.

\label{SS-simple}
  Let $\frak g$ be a semisimple Lie algebra over $\C$ with Chevalley basis $\mathcal B$ as in Theorem~\ref{Cbasis}. Let
$\mathfrak{g}_{\mathbb{Z}}$
be the $\Z$-span of $\mathcal B$.  Then
$$\mathfrak{g}_{\mathbb{Z}}=\mathfrak{h}_{\mathbb{Z}}\oplus\bigoplus_{\a\in\Delta}\Z x_{\a},$$
where $\mathfrak{h}_{\mathbb{Z}} =\mathfrak{h}\cap \mathfrak{g}_{\mathbb{Z}}$ is the $\Z$-span of the set of all corrots $h_{\a}:=[x_{\a},x_{-\a}]$ for $\a\in\Delta$.
Then $\mathfrak{g}_{\mathbb{Z}}$ is a lattice, that is, a free $\Z$-module, and the Lie bracket of any two elements of $\mathfrak{g}_{\mathbb{Z}}$ lies in $\mathfrak{g}_{\mathbb{Z}}$. 
Thus $\mathfrak{g}_{\mathbb{Z}}$ is a Lie ring. 
The subgroup $\mathfrak{h}_{\mathbb{Z}}$ is   a free abelian group with basis $h_i$, where ~$\a_i$ are the simple roots.
Furthermore the canonical map
~$\frak g_{\Z}\otimes_{\Z}\C\to \frak g$ is an isomorphism, so~$\frak g_{\Z}$ is a $\Z$-form of~$\frak g$.

\subsection{The Cartier--Kostant  integral form}\label{CSform}
Fix a Chevalley basis $\mathcal{B}=\{x_{\a},h_{i}\mid \a\in\Delta,\ i\in I \}$.
Let $\mathcal{U}$ be the universal enveloping algebra of $\fg$ and 
identify $\mathfrak{g}$ with its image in $\mathcal{U}$. 
Define
$$\left (\begin{matrix}
h \\ m\end{matrix}\right):= \frac{h(h-1)\dots (h-(m-1))}{m!}\quad\text{and}\quad x_\a^{(m)} := \frac{(x_\a)^m}{m!}$$ 
for $h\in\fh$, $\a\in\Delta$, $m\geq 0$. Let 
\begin{itemize}
\item  ${\mathcal U}^+_{{\mathbb{Z}}}\subseteq {\mathcal U}$ be the ${\mathbb{Z}}$-subalgebra generated by $x^{(m)}_{\a}$ for $\a\in \Delta^+$
and
$m\geq 0$,  
\item  ${\mathcal U}^-_{{\mathbb{Z}}}\subseteq {\mathcal U}$ be the ${\mathbb{Z}}$-subalgebra generated by $x^{(m)}_{\a}$ for $\a\in \Delta^-$
and
$m\geq 0$,  
\item  ${\mathcal U}^0_{{\mathbb{Z}}}\subseteq {\mathcal U}$ be the ${\mathbb{Z}}$-subalgebra generated by $\left (\begin{matrix}
h_i \\ m\end{matrix}\right )$ for $i\in I$ and $m\geq 0$, 
\item ${\mathcal U}_{{\mathbb{Z}}}\subseteq {\mathcal U}$ be the ${\mathbb{Z}}$-subalgebra generated by $\calU^+_\Z$, $\calU^-_\Z$, and $\calU^0_\Z$.
\end{itemize}

We give the positive roots a fixed linear ordering compatible with height. That is, $\htt(\a)<\htt(\b)$ implies $\a<\b$.
\begin{proposition}[\cite{Kostant}]
$$ B^0=\left\{\prod_{i\in I} \binom{h_i}{m}  \;\Bigg\vert\; m_i \in \Z_{\ge0} \right\}$$
is a $\Z$-basis of $\calU_\Z^0$, and
$$B^\pm:= \left\{ \prod_{\a\in\Delta^+} x_{\pm\a}^{(m_\a)} \;\Bigg\vert\; m_\a\in \Z_{\ge0} \right\}$$ 
is a $\Z$-basis of $\calU_\Z^\pm$, where the product is written with respect to the fixed linear ordering.
The multiplication map
$$ \calU_\Z^-\otimes \calU_\Z^0\otimes\calU_\Z^+ \rightarrow \calU_\Z$$
is an isomorphism, and so
$$ \{ b_-b_0b_+ \mid b_-\in B^-, b_0\in B^0, b_+\in B^+\}$$
is a $\Z$-basis of $\calU_\Z$.
\end{proposition}

\begin{theorem}[\cite{Chevalley,Kostant}]
$\calU_\Z$ is a $\Z$-form of $\calU$ and 
$\calU_\Z\cdot \fg_\Z=\fg_\Z$.
\end{theorem}

\subsection{Admissible lattices in $\fg$  and root data} 
As above, we fix a Chevalley basis $\mathcal{B}=\{x_{\a},h_{i}\mid \a\in\Delta,\ i\in I \}$.
We define an {\it admissible lattice} in $\fg$ to be a $\Z$-form of $\fg$ that is preserved by $\calU_\Z $. 
Given a lattice $L$ such that $Q\subseteq L\subseteq P$, we can construct an admissible lattice 
$$\fg_L := L^* + \Z\{x_\a\mid \a\in\Delta\}. $$ 
\begin{theorem}
Every admissible lattice in $\fg$ has the form $\fg_L$ for some $Q\subseteq L\subseteq P$.
\end{theorem}
 We have $\fg_\Z=\fg_Q$, called the \emph{adjoint form}. 
We call $\fg_P$ the \emph{simply connected form}.
See Figure~\ref{F-adj-sc}.
\begin{figure}[h]
\begin{center}
\begin{tikzcd}
P^\vee=L^* \arrow[rrdd, dotted,leftrightarrow] &                & P                    &  &  & P^\vee                                           &                                 & P=L                \\
                                &                &                      &  &  &                                                  &                                 &                    \\
Q^\vee \arrow[uu, hook]         &                & Q=L \arrow[uu, hook] &  &  & Q^\vee=L^* \arrow[uu, hook] \arrow[rruu, dotted,leftrightarrow] &                                 & Q \arrow[uu, hook]\\
                                & \text{adjoint} &                      &  &  &                                                  & 
\substack{\text{simply}\\\text{connected}} &                    
\end{tikzcd}
\end{center}
\caption{Adjoint versus simply connected lattices}\label{F-adj-sc}
\end{figure}
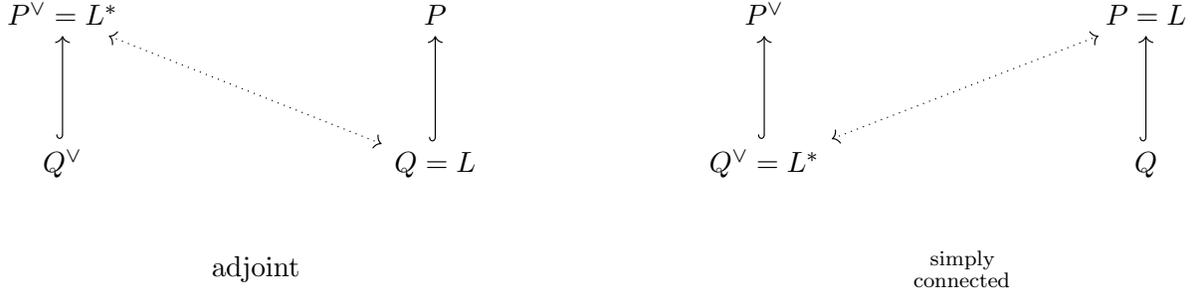
Hence the admissible lattices of $\fg$ are classified by $(L,\Delta,L^*,\Delta^\vee)$. We call this a \emph{root datum}.

For a fixed simple Lie algebra, the forms are classified 
by the subgroups of the \emph{Fundamental group} $P/Q$. By considering the Cartan types case-by-case, we note that
\begin{itemize}
\item For types $E_8,F_4,G_2$, the fundamental group is trivial, so there is only one form, and the adjoint and simply connected cases are isomorphic.
\item For types $A_1, B_\ell,C_\ell, E_7$, the fundamental group has order two, so the only forms are the adjoint and simply connect types.
\item For type $D_\ell$ with $\ell$ even, the fundamental group is a Klein 4 group ($\Z/2\Z\oplus \Z/2\Z$). In all other cases, it is cyclic.
\end{itemize}

\section{Modules and their integral forms}\label{S-hwmodule}
We  assume throughout that 
 $V$ is a finite dimensional faithful highest weight module. This implies that $\fg$ is simple.
 
\subsection{Modules}\label{SS-hwmodule}Let $\mathfrak g$ be a finite--dimensional simple Lie algebra over $\C$ with Chevalley basis $\mathcal{B}=\{x_{\a},h_i\mid \a\in\Delta,\ i\in I \}$.
Suppose that $V$ has underlying  representation $\rho:\fg\rightarrow\End(V)$.
Let $\mu\in \mathfrak h^*$ and set
$$V_\mu=\{v\in V\mid  h\cdot v=\mu(h)v, \text{ for all }h\in \mathfrak{h}\}.$$
We say that $\mu$ is a \emph{weight} of $V$ if $V_\mu\neq 0$. Nonzero elements of $V_\mu$ are called \emph{weight vectors} and $V_\mu$ is called the {\it weight space} corresponding to the weight $\mu$. We write ${\rm wts}(V)$ for the set of all weights of $V$
and we say that $V$ is a \emph{weight module} if
$$V=\bigoplus_{\mu\in\wts(V)} V_\mu.$$
Let $\mu\in\wts(V)$ and let $\a\in\Delta$,  then 
$$\fg_\a \cdot V_\mu\subseteq V_{\mu+\a}.$$
 Given $v\in V_{\mu}$, there exists $v'\in V_{w_{\a}\mu}$ such that $\widetilde{w}_{\a}(t)v=t^{-\langle\mu,\a^{\vee}\rangle } v'$, for all $t\in\Q^{\times}$,  where $\widetilde{w}_{\a}(t)$ as defined by Equation (\ref{rootref}). In particular, $\widetilde{w}_{\a}(t)\cdot V_{\mu}=V_{w_{\a}\mu}$.

\begin{theorem}\label{T-wtomnibus}
Since $V\!$ is finite dimensional,  $V$ is a weight module  and $\wts(V)\subseteq P$.
\end{theorem}
\begin{proof}
By the Corollary to (\cite{Steinberg}, Theorem 3), if $\mu$ is a weight of $V$ and $\alpha$ is a root, then $\mu(h_\alpha)\in\Z$. This implies that every weight of a highest weight representation is integral and hence $V$ is finite dimensional by \cite{Bourbaki4-6}.    
\end{proof}

We write
\begin{alignat*}{3}
Q^+&= \{\a\in Q \mid \langle\a,\omega_i^\vee\rangle\ge0, \text{ for all $i\in I$}\} \;&=\sum_{i\in I}
\Z_{\ge0} \a_i,\\
P^+&= \{\lambda\in P \mid \langle\lambda,h_i\rangle\ge0, \text{ for all $i\in I$}\} &=\sum_{i\in I}
\Z_{\ge0} \omega_i.
\end{alignat*}

The elements of $P^+$ are called \emph{dominant integral weights}.
Recall that the \emph{dominance order} is the partial order on $P$ defined by $\mu\ge\nu$  if{f} $\mu- \nu\in Q^+$.
If $\wts(V)$ contains an element $\lambda$ such that $\lambda\ge\mu$ for all $\mu\in\wts(V)$, 
then $V$ is called a \emph{highest weight module} with \emph{highest weight} $\lambda$.
\begin{theorem}[\cite{Humphreys}]\label{T-wtirredomnibus}
Every finite dimensional irreducible module of $\fg$ is a highest weight module. 
Two irreducible modules are isomorphic if and only if they have the same highest weight.
There is an irreducible  module $V^\lambda$ with highest weight $\lambda$ for every dominant integral  weight $\lambda$.
\end{theorem}

By (\cite{Steinberg}, Theorem 3), the weight space $(V^\lambda)_\lambda \subseteq V^\lambda$ for the highest weight has dimension one. If we fix a nonzero highest weight vector vector $v_{\lambda}$ such
that $v_{\lambda}$ is a weight vector corresponding to some weight $\lambda$ and such that for any positive root $\alpha$, we have  $x_\alpha\cdot v_{\lambda} = 0$, 
then
$$\fn^+\cdot v_\lambda =0,\qquad \fh\cdot v_\lambda = \C v_\lambda,\qquad \fn^-\cdot v_\lambda=V^\lambda,$$
and so
$$\calU^+\cdot v_\lambda =0,\qquad \calU^0\cdot v_\lambda=\C v_\lambda,\qquad\calU^-\cdot v_\lambda=V^\lambda.$$

 Let $\mu\in\wts(V)$. Since $\mu\le \lambda$, we have $\lambda-\mu\in Q^{+}$ and hence $$\lambda-\mu=\sum_{i=1}^{\ell}n_{i}\alpha_{i}$$ with $n_{i}\in \Z_{\ge 0}$ for all $i=1,2,\dots, n.$ The {\it depth} of  $\mu$  is  $$\dep(\mu)= \sum_{i=1}^{\ell}n_{i}=\htt(\lambda)-\htt(\mu).$$

The lattice generated by the weights of $V$ is
$$L_V:= \Z\wts(V)\subseteq P.$$
\begin{lemma}[\cite{Steinberg}, Lemma 27]
Since $V$ is a faithful module, 
 $Q\subseteq L_V$.
\end{lemma}
\begin{proof}
Since $V$ is faithful, no element of $\frak g$ acts trivially on $V$. Let $\a$ be a root. Then the root vector $x_{\a}$ is nonzero and $\a\in Q$. Thus for a weight $\mu$ of $V$ there exists $0\neq v\in V_{\mu}$ such that
$0\neq x_{\a}v\in V_{\mu+\a}$. Hence $\mu+\a$ is also a weight of $V$ and so
$$\a=(\mu+\a)-\mu \in L_V.$$
Thus $Q\leq L_V$. 
\end{proof}

For $\alpha\in \Delta$, let $h_{\a}$ be the coroot corresponding to $\alpha$.  For $n\ge 0$, also recall the element
$$x^{(n)}_{\a} := \frac{{x^{n}_\a}}{n!}$$
of $\mathcal{U}_{\Z}$
as introduced in Subsection~\ref{CSform}.
We end this subsection with the following lemma.
\begin{lemma}\label{L-xx}
Let $\mu\in \wts(V)$be such that $\mu+\alpha$ is not a weight of $V$. 
Further assume that $v_{\mu}\in V_{\mu}$ and $n:= \langle\mu,h_{\a}\rangle>0$.
Then
$$x^{(n)}_{\alpha} x^{(n)}_{-\alpha} \cdot v_{\mu} = v_\mu.$$
\end{lemma}
\begin{proof}
First, we prove by induction on $m\in\N$ that
$$[x_{\alpha},{x^{m}_{-\alpha}}]=m{x^{m-1}_{-\alpha}}\left(h_{\a}-(m-1)\right).$$
 By assuming the inductive hypothesis for 
 $k-1\in\N$, we take
\begin{align*}
[x^{k}_\a,{x_{-\a}}] &= x_\a {x^{k}_{-\a}} -{x^{k}_{-\a}} x_\a \\
&=x^k_\a{x_{-\a}} - {x^{k-1}_{-\a}}x_\a{x_{-\a}} +{x^{k-1}_{-\a}}x_\a{x_{-\a}}   -{x^k_{-\a}} x_\a\\
&= [x_\a,{x^{k-1}_{-\a}}] {x_{-\a}} + {x^{k-1}_{-\a}} [x_\a,{x_{-\a}}]\\
&=  (k-1) {x^{k-2}_{-\alpha}}\left(h_{\a}-k+2 \right){x_{-\a}} 
+ {x^{k-1} _{-\alpha}}h_{\a}\\
&= (k-1) {x^{k-2}_{-\alpha}}(h_{\a}x_{-\a}) -(k-1)(k-2){x^{k-1}_{-\alpha}} + {x^{k-1}_{-\alpha}} h_{\a}\\
&= (k-1) {x^{k-2}_{-\alpha}}(x_{-\a}h_{\a}-2x_{-\a}) -(k-1)(k-2){x^{k-1}_{-\alpha}} + {x^{k-1}_{-\alpha}} h_{\a}\\
&= k {x^{k-1}_{-\alpha}}h_{\a} -k(k-1){x^{k-1}_{-\alpha}}\\
&= kx^{k-1}_{-\alpha}(h_{\a} -(k-1)).
\end{align*}

Secondly, we note that $x_\a\cdot v_\mu\in V_{\mu+\a}=0$, so $x_\a\cdot v_\mu=0$.
Also $h_{\a} \cdot v_\mu=\langle\mu,h_{\a}\rangle v_\mu=nv_\mu$.
Next we prove
\begin{eqnarray*}
{x^{m}_{\alpha}}{x^{m}_{-\alpha}}\cdot v_{\mu}=m!\prod_{j=0}^{m-1}(n-j)\cdot v_{\mu}
\end{eqnarray*}
by induction on $m$:
\begin{align*}
{x^{m}_{\alpha}}{x^{m}_{-\alpha}}\cdot  v_{\mu}
&={x^{m-1}_\a}\left([x_{\alpha}, {x^{m}_{-\alpha}}]+{x^{m}_{-\alpha}}x_{\alpha}\right)\cdot  v_{\mu}\\
&= {x^{m-1}_\a}[x_{\alpha}, {x^{m}_{-\alpha}}]\cdot  v_{\mu}\\
&= {x^{m-1}_{\alpha}}{x^{m-1}_{-\alpha}} m\left(h_{\a}-(m-1)\right)\cdot  v_{\mu}\\
&= m\left(n-(m-1)\right) {x^{m-1}_{\alpha}}{x^{m-1}_{-\alpha}}\cdot  v_{\mu}\\
&= m!\prod_{j=0}^{m-1}(n-j)\cdot v_{\mu}.
\end{align*}
Now take $m=n$ to get
${x^{n}_{\alpha}}{x^{n}_{-\alpha}}\cdot  v_{\mu} = n!\,n!\,v_\mu$, so $x^{(n)}_{\alpha} x^{(n)}_{-\alpha} \cdot v_{\mu} = v_\mu$.
\end{proof}

\subsection{Admissible lattices in modules}\label{S-hwmodule}
An \emph{admissible lattice} $M$ in $V$ is a $\Z$-form of $V$ such that $\calU_\Z\cdot M\subseteq M$.
We consider the case where $\fg$ is simple and $V=V^\lambda$ is a nontrivial irreducible module. It follows that $V$ is faithful and so $Q\subseteq L_V$. Since $\mu-\lambda\in Q$ for all $\mu\in\wts(V)$, we have that $L_V$ is the lattice generated by $\lambda$ and $Q$.

\begin{theorem}[\cite{Chevalley}]
Since $\fg$ is simple and $V=V^\lambda$
is a nontrivial irreducible module
$$V_\Z := \calU_\Z\cdot v_\lambda$$
is an admissible lattice for $V$.

\end{theorem}

\begin{proposition}[\cite{Humphreys}]
If $M$ is an admissible lattice of $V$, then
$V_\Z\subseteq M \subseteq ((V^*)_\Z)^*.$
\end{proposition}

\begin{lemma}[\cite{Steinberg}, Lemma 12]\label{lattwsd}  Let $v_{\lambda}$ be a highest weight vector. Then 
$\mathcal{U}_\Z^0 \cdot v_{\lambda}=\Z v_{\lambda}$.
For $\mu\in\wts(V)$, $V_{\mu,\Z}:= V_\Z\cap V_\mu$ is a $\Z$-form of $V_\mu$. Hence
\begin{eqnarray*}
 V_\Z =\bigoplus_{\mu\in\wts(V)} V_{\mu,\Z}.
\end{eqnarray*}
\end{lemma}

We have excluded the trivial module, but this case is straightforward: If $V\cong\C$
is the trivial module, then any nonzero $v_0\in V$ is a highest weight vector for the weight $\lambda=0$
and $V_\Z=\Z v_0$ is admissible.

\begin{longver}
Finally we consider  $\fg$  semisimple and $V$  an arbitrary finite-dimensional module.
Then $V$ is also semisimple, and so
 $$V =\bigoplus_{j=1}^m V_j$$
where each $V_j\cong V^{\lambda_j}$ for some $\lambda_1,\dots,\lambda_m\in P^+$.
If we choose highest weight vectors $v_j\in V_j$, then $V_{j,\Z}:= \calU_Z \cdot v_j$ is an admissible lattice in $V_j$, and
$$V_\Z =\bigoplus_{j=1}^m V_{j,\Z}$$
is an admissible lattice in $V$.
\end{longver}

\subsection{The $\Z$-form of $\fg$ preserving $M$}
Let $M$ be an admissible lattice in a highest weight module $V$ of a simple Lie algebra $\fg$.
Then 
$$\mathfrak g_{M}:=\{x\in\fg \mid x\cdot M\subseteq M\}$$
is an admissible form in $\fg$.
\begin{proposition}[\cite{Steinberg}, Corollary 2 of Lemma 13]
Since $\fg$ is simple and $V$ is irreducible,
$$ g_M = L_V^* \oplus \bigoplus_{\a\in\Delta}\Z\a =\fg_{L_V},$$
where 
$$L^*_V :=\{ h\in\mathfrak{h}\mid\langle\mu,h\rangle\in\Z\text{ for all } \mu\in{\rm wts}(V) \}.$$
\end{proposition}
In particular, $\mathfrak g_{M}$ is independent of $M$, but is not independent of $V$.

Note that the adjoint module $V=\fg$ is a highest weight module whose highest weight is the highest root of $\Delta$.
In this case $L_V=Q$, so $\fg_{L_V}$ is the adjoint form $\fg_Q$.

\begin{lemma}
Let $\fg$ be simple, and fix a lattice $Q\le L\le P$.
Suppose   $L/Q$ is cyclic.
Then we can find a dominant weight $\lambda$
so that $L$ is generated by $Q$ and $\lambda$, and so $L=L_V$ where $V=V^\lambda$.
\end{lemma}
Note that this covers all cases, except for  simply connected forms of type $D_\ell$ with $\ell$ even.
In these exceptional cases, a reducible module $V$ is required.

\section{Chevalley groups}\label{universal}
 Let $\mathfrak g$ be a complex simple Lie algebra and let $(V,\rho)$ be a finite dimensional faithful $\mathfrak g$-module. The following automorphisms of $V$ will play a crucial role in defining Chevalley groups.

\subsection{The Chevalley group over $\C$}
Recall that $x\in\frak{g}$ acts \emph{nilpotently} on $V$ if, for each $v\in V$, there is a natural number $n$ such that
$\rho(x)^{n}\cdot v=0$. Since $V$ is finite dimensional, any $x\in\frak{g}$ is automatically nilpotent.
It follows that
$$\chi_{\a}(t):=\exp(t\cdot \rho(x_{\a}))=I+t \rho(x_{\a})+\frac{t^2}{2}(\rho(x_{\a}))^2+\cdots$$ 
is a finite sum,  and hence an element of $\Aut_\fg(V)$. The \emph{root group} of $\a$ is now
$$U_\a= \{\chi_\a(t)\mid t \in \C^\times\}.$$

Given $\varpi\in P$ and $t\in\C^\times$, we  define $h_{\varpi}(t)$  on  $V=\bigoplus_{\mu\in\wts(V)} V_\mu$ by
$$h_{\varpi}(t)\cdot v = t^{\langle\mu, \varpi\rangle} v $$
for $v\in V_\mu$.
Then $h_{\varpi}(t)$ is also in $\Aut_\fg(V)$ and we define the \emph{toral group}
$$H_V=\langle h_{\varpi}(t)\mid \varpi\in L_V^*,\, t\in \C^\times\rangle.$$

Together these generate the  \emph{Chevalley group over $\Q$}  
$$G_V:=\langle H_V,\ U_{\alpha}\mid \a\in\Delta\rangle.$$
Moreover, $G_V\subseteq\GL(V)$ is a simple linear algebraic group, $H_V$ is a maximal torus of $G_V$, and the lattice $L_V$ can be realized as  the lattice ${\rm Hom}(H_, S^1
)$ of characters of $H_V$ \cite{Steinberg}. 

For $\alpha\in\Delta$ and  $s\in\Q^\times$, let  
 \begin{eqnarray}\label{rootref}
 \widetilde{w}_{\alpha}(s)= \chi_{\alpha}(s)\chi_{-\alpha}(-s^{-1})\chi_{\alpha}(s).
\end{eqnarray}
 For each $i\in I$, set $\widetilde{w}_{i}:=\widetilde{w}_{\alpha_{i}}=\widetilde{w}_{\alpha_{i}}(1)$. 
 We write $\widetilde{W}$ for the subgroup of $\Aut(V)$ generated by $\widetilde{w}_{i}$, for all $i\in I.$ 
Each $w\in W$ has a reduced decomposition $w_{j_{1}}w_{j_{2}}\dots w_{j_{m}}$, we denote the corresponding element 
\begin{eqnarray*}
\widetilde{w}=\widetilde{w}_{j_{1}}\widetilde{w}_{j_{2}}\dots \widetilde{w}_{j_{m}} \in\widetilde{W}.
\end{eqnarray*}
\begin{lemma}\label{ww} 
We have
$\widetilde{w}\cdot V_\mu=V_{w\mu}.$
In particular $\dim V_{w\mu}=\dim V_\mu$. 
\end{lemma}

By Lemma~\ref{weroot}, for each $\a\in\Delta$, 
 we have  $\alpha=w\alpha_i$, for some $w\in W$. 
It follows that
$$\chi_{\alpha}(t)=\exp(t\rho( x_{\a}))= \widetilde{w}\chi_{\alpha_i}(\pm t)\widetilde{w}^{-1}.$$

If $ L=Q$ then $G_V(\C)$ is called an \emph{adjoint Chevalley group}. If $ L=P$,  then $G_V(\C)$ is called a \emph{simply connected Chevalley group}.
If $Q$ and $P$ are isomorphic, then the adjoint and simply connected groups coincide.
  
\begin{proposition}[\cite{Steinberg}]\label{sl2s'}  For each $i\in I$ there is a homomorphism 
$$\varphi_i:\SL_2(\C)\longrightarrow G_V(\C)$$
with
$$\varphi_i\left(\begin{matrix} 1 & s\\ 0 & 1\end{matrix}\right)\mapsto \chi_{\alpha_i}(s),\ s\in \C,$$
$$\varphi_i\left(\begin{matrix} 1 & 0\\ t & 1\end{matrix}\right)\mapsto \chi_{-\alpha_i}(t),\ t\in \C.$$
\end{proposition}

The following theorem  shows that the isomorphism type of $G_V(\C)$ is independent of the choice of representation.
\begin{theorem}[\cite{Steinberg}, Corollary 5]\label{uniqueness}  Let $V_1$ and $V_2$ be faithful $\fg$-modules with $L_{V_1}=L_{V_2}$,
then $G_{V_1}({\C})\cong G_{V_2}({\C}).$
\end{theorem}

\subsection{Forms of $G_V$}
We continue to use the notation from Section~\ref{S-hwmodule}.
So $V_{\mathbb{Z}}=\calU_\Z\cdot v_\lambda$ is a $\Z$-form of $V$ such that ${\mathcal U}_{{\Z}}\cdot V_\Z\subseteq V_\Z$.
Let $\rho_\Z$ be the restriction of $\rho$ to $V_\Z$
So $V_\Z$ is also a $\calU_\Z$-module and $\rho_\Z$ extends to a representation of $\calU_\Z$.
We can define
$$\exp(tx_\a) := 1 + tx_\a^{(1)} + tx_\a^{(2)} + \cdots$$
as an element of  the ring 
$\mathcal{U}_\Z[[t]]$ of formal power series in variable
$t$ over $\mathcal{U}_\Z$.  
Specializing to $t\in\Z$ we get
$$\chi_\a(t) =\rho_\Z(\exp(tx_\a))\in\GL(V_\Z). $$
So we can define
\begin{align*}
    U_\a(\Z) &:= \{\chi_\a(t)\mid t \in \Z^\},\\
    H_V(\Z)&:=\langle h_{\varpi}(t)\mid \varpi\in L_V^*,\, t\in \Z^\times\rangle,\text{ and}\\
    G_V(\Z)&:=\langle H_V(\Z),\ U_{\alpha}(\Z)\mid \a\in\Delta\rangle.
\end{align*}
Note that $\Z^\times=\{\pm1\}$, so $H_V(\Z)=\langle h_{\varpi}(-1)\mid \varpi\in L_V^*\rangle$.

Let $\a,\b$ be roots with $\a+\b\neq 0$. Then in the ring 
$\mathcal{U}_\Z[[t, u]]$ of formal power series in
$t, u$ over $\mathcal{U}_\Z$,  by \cite{Steinberg} we have 
$$(\exp(t\rho(x_\a)),\exp(u\rho(x_\b)))=\prod_{i,j>0}\exp(c_{ij}t^iu^j \rho(x_{i\a+j\b}))$$
where $(g,h)=ghg^{-1}h^{-1}$, the product on the right is taken over all roots $i\a+j\b$ arranged in some fixed order for $i,j>0$, and  the $c_{ij}$s are integers depending on $\a,\b$
and the chosen ordering, but not on $t$ or $u$. Moreover,  $c_{11} = p_{\a\b}$.

Let $U(\C)$ be the group generated by all $\chi_{\alpha}(t)$ for $\a\in\Delta^+$ and all $t\in\C$. 
It follows that
$$ U(\Z) := \langle \chi_{\a_i}(t) \mid t\in\Z\rangle$$
is a nilpotent group containing $U_\a(\Z)$ for all positive roots $\a$.

\begin{longver}
    For example, $L=Q$ if $V$ is the adjoint representation and $L=P$ if $V$ is the direct sum of representations having fundamental weights as
dominant weights. In this case, $V=V^{\omega_1}\oplus\dots \oplus V^{\omega_{\ell}}$ is not irreducible, but is the direct sum of irreducible highest weight modules and $\wts(V)=\bigcup_{i=1}^{\ell} \wts(V^{\omega_i}),$ hence the lattice generated by $\wts(V)$ is the weight lattice $P$. 

If $\lambda=\sum_{i}a_i\omega_i$ and $a_i\geq 0$ for each $i$ then the highest weight module $V^\lambda$ with highest weight $\lambda$ is an irreducible module.

\end{longver}

Finally we can define groups over an arbitrary commutative ring $\K$. Note that we only use $\K=\Z,\Q,\C$ in this paper.
Set $V_\K:=\K\otimes_{\Z}V_\Z$, $\fg_\K:=\K\otimes_\Z\fg_\Z$, and $\calU_\K:= \K\otimes_\Z \calU_\Z$.
Then $V_\K$ is a $\calU_\K$ module with representation $\rho_\K: \calU_\K\rightarrow\End(V_\K)$.
Also  $\chi_{\alpha}(a)$ and $h_\varpi(t)$ are elements of $\Aut_{\fg_\K}(V_\K)$ for all
$\a\in\Delta$, $\varpi\in L_V$, $s\in\K$, $t\in\K^\times$. Define
\begin{align*}
    U_\a(\K) &:= \{\chi_\a(t)\mid t \in \K\},\\
    H_V(\K)&:=\langle h_{\varpi}(t)\mid \varpi\in L_V^*,\, t\in \K^\times\rangle,\text{ and}\\
    G_V(\K)&:=\langle H_V(\K),\ U_{\alpha}(\K)\mid \a\in\Delta\rangle.
\end{align*}
For $\K=\Z$, this coincides with the definitions given earlier in this section.
For $\K=\C$, $U_\a(\C)=U_a$, $H_V(\C)=H_V$, and $G_V(\C)=G_V$.
We also have
$$ U(\K) := \langle \chi_{\a_i}(t) \mid t\in\K\rangle$$
is a nilpotent group containing $U_\a(\K)$ for all positive roots $\a$.

\section{Example: Integrality in $\SL_2(\Q)$}\label{integrality} 

Let $\mathfrak g$ be the Lie algebra $\mathfrak{sl}_2$ of  $2\times 2$ matrices of trace zero with Chevalley basis $$x_{\a}=\left(\begin{matrix} 0 & 1\\ 0 & 0\end{matrix}\right),\ x_{-\a}=\left(\begin{matrix} 0 & 0\\ 1 & 0\end{matrix}\right),\ h_{\a}=\left(\begin{matrix} 1 & 0\\ 0 & -1\end{matrix}\right).$$

There is  one positive root $\alpha$ and one fundamental weight $\omega=\frac{1}{2}\alpha$. We have $\langle\omega,\alpha\rangle=2$ and $\alpha$ and $\omega$ are both dominant integral weights. Thus
$$Q=\Z\alpha\text{ and } P=\Z\omega=\frac{1}{2}\Z\alpha,$$
thus
$P/Q\cong\mathbb{Z}/2\mathbb{Z}$.

\begin{lemma} The standard representation $V_{\Q}=\Q\oplus\Q$ can be realized as a faithful highest weight representation $\rho$ of $\mathfrak{sl}_2$ with highest weight   $\omega$. The subset $V_{\Z}=\Z\oplus\Z$ is an admissible lattice in $V_{\Q}$.
\end{lemma}

\begin{proof}  The representation $V_{\Q}=\Q\oplus\Q$ is a faithful $\mathfrak{sl}_2$-module.
The weights of $V_{\Q}$ are $-\omega$, $0$, $\omega$ with $\langle\omega,h_\alpha\rangle=1$, so $\omega$ is the highest weight.

Let $\mathcal{U}_{\Z}$ denote the Cartier--Kostant $\Z$-form of the universal enveloping algebra $\mathcal{U}(\Q)$ of $\mathfrak{sl}_2$. Then  $\mathcal{U}(\Q)$ has $\Z$-subalgebras $\mathcal{U}^-_{\Z}$, $\mathcal{U}^0_{\Z}$ and $\mathcal{U}^+_{\Z}$ with $\Z$-bases consisting of $\dfrac{(x_{-\a})^m}{m!}$,  $\left (\begin{matrix} h_{\a}\\ b\end{matrix}\right )$ and $\dfrac{(x_{\a})^n}{n!}$ respectively, for $m,n,b\in\Z_{\geq 0}$ such that
$\mathcal{U}_{\Z}=\mathcal{U}^-_{\Z}\mathcal{U}^0_{\Z}\mathcal{U}^+_{\Z}.$
We have
$$V_{\Z}=\Z\left(\begin{matrix} 1 \\ 0\end{matrix}\right)\oplus \Z\left(\begin{matrix} 0 \\ 1\end{matrix}\right)=
V_{\omega}\oplus V_{-\omega}$$
and 
$$x_{-\a}\cdot v_\omega=\left(\begin{matrix}  0 & 0 \\ 1 & 0  \end{matrix}\right)\left(\begin{matrix} 1 \\ 0 \end{matrix}\right)= 
\left(\begin{matrix} 0 \\ 1\end{matrix}\right).$$
That is, $$x_{-\a}:V_{\omega}\mapsto V_{-\omega}.$$
Since $x_{\a}^2=0$,  $$\dfrac{(x_{-\a})^m}{m!}\cdot V_{\Z} \subseteq  V_{\Z},$$
for $m\in {\Z}_{\geq 0}$. 
Moreover
$$\mathcal{U}_{\Z}\cdot v_{\omega}=\mathcal{U}^-_{\Z}\cdot \Z v_{\omega}$$
since all non-trivial elements of $\calU^+_\Z$  annihilate $v_{\omega}$ and $\left (\begin{matrix} h_{\a}\\ b\end{matrix}\right )$ acts as scalar multiplication on $v_{\omega}$ by a $\Z$-valued scalar (\cite{Humphreys}, Theorem 27.1) for $b\in\Z_{\geq 0}$. Thus $\Z\oplus\Z$ is an admissible lattice in~$V_\Q$.
\end{proof}

The group  $G(\mathbb{Q})$ is the simply connected group $\SL_2(\mathbb{\Q})$. That is, 
\begin{align*}  
G(\mathbb{Q})&=\langle\chi_{\a}(s),\chi_{-\a}(t)\mid s,t\in\Q\rangle\\
&=\langle \exp(s\rho(x_{\a})),\ \exp(t\rho(x_{-\a}))\mid s,t\in\Q\rangle\\
&=\left\langle \left(\begin{matrix} 1 & s\\ 0 & 1\end{matrix}\right),\ \left(\begin{matrix} 1 & 0\\ t & 1\end{matrix}\right)\mid s,t\in\Q\right\rangle\\
&=\SL_2(\Q).
\end{align*}

Moreover $\SL_2(\Q)$  acts on $V_{\Q}$:
$\left(\begin{matrix} a & b\\ c & d\end{matrix}\right)\cdot \left(\begin{matrix} x \\ y \end{matrix}\right)=
\left(\begin{matrix} ax+ by\\ cx +dy\end{matrix}\right).$

We note that  $G(\Z)=\SL_2(\mathbb{Z})$ is generated by $\chi_{\alpha}(s)$ and $\chi_{-\alpha}(t)$ for $s,t\in\Z$. 

\begin{proposition} The subgroup $\SL_2(\Z)$ of $\SL_2(\Q)$ is the stabilizer of $V_\Z=\Z\oplus\Z$ in  the standard representation $V_{\Q}=\Q\oplus\Q$ of $\mathfrak{sl}_2(\Q)$.
\end{proposition}

\begin{proof}

Suppose that 
$\left(\begin{matrix} a & b\\ c & d\end{matrix}\right)\in \SL_2(\Q)$ stabilizes $V_{\Z}$:
$$\left(\begin{matrix} a & b\\ c & d\end{matrix}\right)\cdot V_{\Z}= V_{\Z}.$$
That is,  for each $\left(\begin{matrix} x \\ y \end{matrix}\right)\in V_{\Z}$, there exists $\left(\begin{matrix} u \\ v \end{matrix}\right)\in V_{\Z}$ such that
$$\left(\begin{matrix} a & b\\ c & d\end{matrix}\right)\cdot \left(\begin{matrix} x \\ y \end{matrix}\right)=
\left(\begin{matrix} u \\ v \end{matrix}\right).$$
Take $\left(\begin{matrix} x \\ y \end{matrix}\right)=
\left(\begin{matrix} 0 \\ 1 \end{matrix}\right)$. Then  $u=ax+ by$ implies $b\in\Z$ and $v=cx+ dy$ implies $d\in\Z$. Take 
 $\left(\begin{matrix} x \\ y \end{matrix}\right)=
\left(\begin{matrix} 1 \\ 0 \end{matrix}\right)$. Then $u=ax+ by$ implies $a\in\Z$ and $v=cx+ dy$ implies $c\in\Z$. Thus if 
$\left(\begin{matrix} a & b\\ c & d\end{matrix}\right)\cdot V_{\Z}= V_{\Z}$ then $\left(\begin{matrix} a & b\\ c & d\end{matrix}\right)\in \SL_2(\Z)$.
\end{proof} 

\bigskip
To see some of the difficulties that arise in studying Chevalley groups over $\Z$, let $u\in \Q$,
$$\chi_{\alpha}(u)=\left(\begin{matrix} 1 & u\\ 0 &1\end{matrix}\right),\ \chi_{-\alpha}(u)=\left(\begin{matrix} 1 & 0\\ u &1\end{matrix}\right)$$
 and $$\widetilde{w}_{\alpha}(u)= \chi_{\alpha}(u)\chi_{-\alpha}(-u^{-1})\chi_{\alpha}(u),\quad h_{\alpha}(u)=\widetilde{w}_{\alpha}(u) \widetilde{w}_{\alpha}(1)^{-1}.$$
 
Then $\gamma=\left(\begin{matrix} 1 & 1\\ 1 & 2\end{matrix}\right)\in\SL_2(\Q)$ and 
$\gamma\cdot(\Z\oplus\Z)=\Z\oplus\Z$,
 but
 
 $$\left(\begin{matrix} 1 & 1\\ 1 & 2\end{matrix}\right)=\chi_{\a}\left(\frac{1}{2}\right) h_{\a}\left(\frac{1}{2}\right)\chi_{-\a}\left(\frac{1}{2}\right)=
\left(\begin{matrix} 1 & \frac{1}{2}\\ 0 & 1\end{matrix}\right)\left(\begin{matrix} \frac{1}{2} & 0\\ 0 & 2\end{matrix}\right)
\left(\begin{matrix} 1 & 0\\ \frac{1}{2} & 1\end{matrix}\right).$$


However, since $\SL_2(\mathbb{Z})$ is generated by $\chi_{\alpha}(t)$ and $\chi_{-\alpha}(t)$ for $t\in\Z$, we may look for another decomposition of $\left(\begin{matrix} 1 & 1\\ 1 & 2\end{matrix}\right)$ in terms of generators whose scalars are integer valued, and we find
$$\left(\begin{matrix} 1 & 1\\ 1 & 2\end{matrix}\right)=\chi_{-\a}(1) \chi_{\a}(1)=
\left(\begin{matrix} 1 & 0\\ 1 & 1\end{matrix}\right)\left(\begin{matrix} 1 & 1\\ 0 & 1\end{matrix}\right).$$
Thus given $\left(\begin{matrix} a & b\\ c & d\end{matrix}\right)\in \SL_2(\Z)$, {there exist} $t_1,\dots, t_k\in\Z$ such that 
$$ \left(\begin{matrix} a & b\\ c & d\end{matrix}\right)=\chi_{\pm\a}(t_1)\chi_{\pm\a}(t_2)\dots \chi_{\pm\a}(t_k).$$


\section{Integrality of the unipotent subgroup}\label{S-integrality}
As before, let $\mathfrak{g}$ be a finite dimensional complex simple Lie algebra with a fixed Chevalley basis and let $(V,\rho)$ a faithful finite dimensional highest weight representation of $\mathfrak{g}$ with highest weight $\lambda\in P$. We will also consider the extension of $\rho$ to the universal enveloping algebra $\mathcal{U}$ of $\fg$. For each $\alpha\in\Delta$, let $x_{\alpha}\in\mathfrak{g}_{\alpha}$ be the Chevalley basis vector.   Recall that 
$U_{\alpha}(\Q)=\{\chi_{\alpha}(s)\mid s\in \mathbb{Q}\}$
is the root group corresponding to $\alpha$,
where for 
$s\in \mathbb{Q}$, \begin{eqnarray}
\chi_{\alpha}(s)&=&\exp(s \rho(x_{\a}))=1+s\rho(x_{\alpha})+s^{2}\frac{\rho(x_{\alpha})^{2}}{2!} +\dots+s^{k}\frac{\rho(x_{\alpha})^{k}}{k!} \nonumber\\
&=&1+s\rho(x_{\alpha}^{(1)})+s^{2}\rho(x_{\alpha}^{(2)})+\dots+s^{k} \rho(x_{\alpha}^{(k)})\label{rootexp}
\end{eqnarray}
for some non-negative integer $k$. From now on, we fix $(V,\rho)$ and drop the notation $\rho$ from the expansion of $\chi_{\alpha}(s)$ to write 
\begin{eqnarray}
\chi_{\alpha}(s)
&=&1+sx_{\alpha}^{(1)}+s^{2}x_{\alpha}^{(2)}+\dots+s^{k} x_{\alpha}^{(k)}.\label{rootexp2}
\end{eqnarray}


\begin{lemma}[\cite{Steinberg}, Lemma 17, Corollary 2]  Every element of $U({\mathbb{Q}})$ can be written uniquely in the form
$$\prod_{\alpha\in\Delta^+} \chi_{\alpha}(t_{\alpha})$$
for some $t_{\alpha}\in\mathbb{Q}$
where the product is taken relative to any fixed order on the positive roots.
\end{lemma}

Recall that we fix a linear order $\leq $ on the set of positive roots $\Delta^+$ that is compatible with  height.
With this fixed order, we write each $ u\in U(\mathbb{Q})$ uniquely as
\begin{eqnarray*}u=\prod_{\i=1}^{n} \chi_{\beta_{j_i}}(t_{\beta_{j_i}}),
\end{eqnarray*}
where $n=|\Delta^+|$, $t_{\beta_{j_i}}\in \Q$ and $\beta_{j_i}\in \Delta_{+}$ are such that
\begin{eqnarray}\label{lop}
 {\rm ht}(\beta_{j_1})\le {\rm ht}(\beta_{j_2})\le \dots \le {\rm ht}(\beta_{j_n}).   
\end{eqnarray}

\subsection{Integrality of root groups}
We start with the following lemma which follows from \cite[Lemma 11]{Steinberg}.

\begin{lemma}\label{rootpower}
Let $\alpha\in \Delta$, and let $v_\mu\in V$ be a weight vector with weight $\mu$. For a non-negative integer $m$, we have
$x_{\alpha}^m\cdot v_{\mu}\in V_{\mu+m\alpha}.$
\end{lemma}

\begin{corollary}
If $\mu=\lambda-m\alpha\in {\rm wts}(V)$, then $x_{\alpha}^m\cdot v_{\mu}=c v_{\lambda}$, for some $c\in\C$.
\end{corollary}
\begin{proof}
Since $\mu$ is a weight of $V$ and $\mu-m\alpha=\lambda$, $x_{\alpha}^m\cdot v_{\mu}\ne 0$. By  Lemma~\ref{rootpower}, $$x_{\alpha}^m\cdot v_{\mu}\in V_{\mu+m\alpha}=V_{\lambda}.$$ The weight space $V_{\lambda}$, for the highest weight $\lambda$, has dimension one and is spanned by $v_{\lambda}$. Thus, $x_{\alpha}^m\cdot v_{\mu}=cv_{\lambda}$ for some $c\in \C.$
\end{proof}

\begin{lemma}\label{stabVZ} For $\alpha\in \Delta$ and for $s\in\mathbb{Z}$, we have 
$\chi_{\alpha}(s)\cdot V_{\mathbb{Z}}\subseteq V_{\mathbb{Z}}$.
Hence
$G(\Z)\cdot V_{\mathbb{Z}}\subseteq V_{\mathbb{Z}}$.
\end{lemma}
\begin{proof}
Since $x_{\a}^{(m)}\in \mathcal{U}_{\mathbb{Z}}$ and $V_{\mathbb{Z}}$ is a $\mathcal{U}_{\mathbb{Z}}$-module,  we have $x_{\a}^{(m)} \cdot V_{\mathbb{Z}}\subseteq V_{\mathbb{Z}}$ for all $m\ge 0$.  Consequently, for all $v\in V_{\Z}$, 
\begin{eqnarray*}\chi_{\alpha}(s)\cdot v&=&(1+sx_{\alpha}^{(1)}+s^{2}x_{\alpha}^{(2)}+\dots+s^{k} x_{\alpha}^{(k)})\cdot v\\
&=&(v+sx_{\alpha}^{(1)}\cdot v+s^{2}x_{\alpha}^{(2)}\cdot v+\dots+s^{k} x_{\alpha}^{(k)}\cdot v)\in V_{\Z}.
\end{eqnarray*}

\end{proof}

 \begin{lemma} \label{rgi}
    
For each $\alpha\in\Delta^{+}$, $\chi_\alpha(s)\cdot V_{\Z}\subseteq V_{\Z}$ if and only if $s\in\Z.$

\begin{proof}
 If $s\in \Z$, then 
 $\chi_\alpha(s)\cdot V_{\Z}\subseteq V_{\Z}$
by Lemma~\ref{stabVZ}. 

Conversely, suppose that $\chi_\alpha(s)$ stabilizes the lattice $V_{\Z}$. More precisely, for all $v\in V_\Z$,
$\chi_\alpha(s)\cdot v\in V_\Z$.
In particular, for all $n\ge 0$
\begin{eqnarray}
\chi_\alpha(s)\cdot (x^{(n)}_{-\alpha}v_{\lambda})\in V_\Z.  
\end{eqnarray}
Let  $n=\langle \lambda, h_{\alpha}\rangle$, then 
\begin{eqnarray*}\chi_{\alpha}(s)\cdot (x^{(n)}_{-\alpha}v_{\lambda})&=&(1+sx_{\alpha}^{(1)}+s^{2}x_{\alpha}^{(2)}+\dots+s^{n} x_{\alpha}^{(n)})\cdot (x^{(n)}_{-\alpha}v_{\lambda})\\
&=&x^{(n)}_{-\alpha}v_{\lambda}+sx_{\alpha}^{(1)}\cdot (x^{(n)}_{-\alpha}v_{\lambda})+s^{2}x_{\alpha}^{(2)}\cdot (x^{(n)}_{-\alpha}v_{\lambda})+\dots+s^{n} x_{\alpha}^{(n)}\cdot (x^{(n)}_{-\alpha}v_{\lambda})).
\end{eqnarray*}
By the assumption,
$$s^{n} x_{\alpha}^{(n)}\cdot (x^{(n)}_{-\alpha}v_{\lambda})= s^{n} (x_{\alpha}^{(n)} x^{(n)}_{-\alpha}\cdot v_{\lambda})\in V_{\Z}.$$
by Lemma~\ref{L-xx},
$s^{n} (x_{\alpha}^{(n)} x^{(n)}_{-\alpha}\cdot v_{\lambda})=s^{n} v_{\lambda}\in V_{\Z}$ and hence $s^n\in\Z$. Consequently, $s\in\Z$.
\end{proof}

\end{lemma}

\subsection{Integrality of $U(\Q)$} As above, we let $\mathfrak{g}$ be a  complex simple Lie algebra and $(V,\rho)$ a faithful finite dimensional highest weight representation of $\mathfrak{g}$ with highest weight $\lambda\in P$.

 Let 
\begin{eqnarray}
    W'=\langle w_{i}\mid i\in I\setminus\{j_{1}\}\rangle 
\end{eqnarray}
be a subgroup of $W$, where $j_{1}$ 
is the index of positive root $\beta_{j_{1}}$ as given in our ordered set  of positive roots $\beta_{j_i}\in \Delta_{+}$, $i=1,\dots, n$.

Set
$$\Omega:=W'\cdot \{\lambda\}.$$
The set $\Omega$ satisfies the following properties. Since $\lambda\in\Omega$, $\Omega\ne \emptyset$. Since $W$ is finite, $\Omega$ is  finite. Finally, $\Omega$ is an ordered subset of $\wts(V)$, as it inherits the order $\le$ from $\wts{V}$.

The set $\Omega$ has a minimal element, say $\nu$, under $\le$. That is, $x\nless\nu$, for all $x\in\Omega$. The definition of $\Omega$ and the minimality of $\nu$  imply that  $\nu=w_0'\lambda$, for some $1\ne w_{0}'\in W'$.

We will need the following lemmas.
\begin{lemma}\label{welong}Given $\nu$ and $w_{0}'$ as above, 
$w_0'$ is the longest element of $W'$. That is, 
$\ell(w_{i}w_0')\le \ell(w_0')$,
for all $i\in I\setminus\{j_{1}\}$, where $\ell$ is the length function on $W$.\end{lemma}
\begin{proof}
Suppose, on the contrary, that there exists $i\in I\setminus\{j_{1}\}$ such that 
$\ell(w_{i}w_0')\le \ell(w_0')+1.$

We have $w_{i}\nu\in\Omega$. Moreover,
\begin{eqnarray*}
w_{i}\nu&=&\nu-\langle \nu, h_{\alpha_{i}}\rangle \alpha_{i}\\ 
&=&\nu-\langle w_0'\lambda, h_{\alpha_{i}}\rangle {\alpha_{i}}\\
&=&\nu-\langle \lambda, (w_0')^{-1}(h_{\alpha_{i}})\rangle {\alpha_{i}}.
\end{eqnarray*}
By assumption, $(w_0')^{-1}(h_{\alpha_{i}})\in Q^{\vee}_{\geq 0}$ (Since each reduced decomposition
$w_0'=w_{k_{1}}w_{k_{2}}\dots w_{k_{m}}$
of $w_0'$
satisfies the property that $w_{k_{1}}\ne w_{i}$ and hence $(w_0')^{-1}= w_{k_{m}} \dots w_{k_{2}} w_{k_{1}}$.
Since $\lambda$ is regular and dominant, we have  $$\langle \lambda, (w_0')^{-1}(h_{\alpha_{i}})\rangle\in \Z_{>0}.$$
Hence $\nu-w_{i}\nu=\langle \lambda, ({w_0'})^{-1}(h_{\alpha_{i}})\rangle {\alpha_{i}}\in Q^{+}$ and thus $w_{i}\nu\le \nu$. This contradicts the minimality of $\nu$.
\end{proof}
 
As before, let $\nu$ be the minimal element of $\Omega$. Since $\nu=w_0'\lambda$, by Lemma~\ref{ww}, $$\dim(V_{\nu})=\dim(V_{\lambda})=1.$$ We fix a basis element $v_{\nu}\in V_{\nu}$ and set 
\begin{eqnarray}\label{nnu0}
n'=\langle \nu,h_{\beta_{j_1}}\rangle
\end{eqnarray}
and
\begin{eqnarray}\label{nnu}
\mu=\nu-n'\beta_{j_1}=w_{\beta_{j_1}}\nu
\end{eqnarray} where we identify $\mu$ with its restriction to $\C h_{\beta_{j_{1}}}$.

\begin{lemma}\label{lowest} The weight $\mu$ is the lowest weight of an $\mathfrak{sl}_{\beta_{j_1}}$-module.
\end{lemma}
\begin{proof}
 Indeed, if $\mu-\beta_{j_1}$ is a weight, then $w(\mu-\beta_{j_1})$ is also a weight for all $w\in W$. In particular, $w_{0}^{-1}(\mu-\beta_{j_1})$ is a weight, where $w_{0}=w_{\beta_{j_1}}w_{0}'$ and $w_{0}'$ as before. We note, though, that
 \begin{eqnarray*}
  w_{0}^{-1}(\mu-\beta_{j_1})&=& \lambda+w_{0}'\beta_{j_1} 
 \end{eqnarray*}
 and $\lambda+w_{0}'\beta_{j_1}\nless \lambda$, which is a contradiction. Hence $\mu-\beta_{j_1}$ is not a weight of $V$ so it cannot be a weight of an $\mathfrak{sl}_{\beta_{j_1}}$-module. 
\end{proof} 

The following theorem shows that $U(\Q)\cap \Gamma(\Z)=U(\Z)$.
\begin{theorem}
\label{stabilizer0}  Let $\mathfrak g$ be a simple Lie algebra. Let $V$ be a finite dimensional faithful highest weight $\fg$-module  with highest weight $\lambda$. Let $v_{\lambda}$ be a highest weight vector. Let $V_{\Z}$ be the admissible lattice of $V$ generated by $v_{\lambda}$. For $s_{j_{i}}\in \Q$, $i=1,\dots,n$, let $u=\prod\limits_{i=1}^{n} \chi_{\beta_{j_i}}(s_{j_i})$ be the unique expression for $u$ in $U(\mathbb{Q})$.  If $u\cdot V_{\Z}\subseteq V_{\Z}$ then $s_{j_i}\in\Z$, for all $1\le i\le n$.
\end{theorem}

\begin{proof}
    We will prove the assertion by  induction on $n$. The base case, $n=1$, follows from Lemma~\ref{rgi} and we have $s_{j_{1}}\in\Z$.

 For the inductive hypothesis, we suppose that every element in $U$ of length $n-1$ satisfies the statement of the theorem. 
 \\
We next suppose that 
$u=\prod\limits_{i=1}^{n} \chi_{\beta_{j_i}}(s_{{j_i}})\in U(\mathbb{Q})$ satisfies $u\cdot V_{\Z}\subseteq V_{\Z}$. Set $u''=\prod\limits_{i=2}^{n} \chi_{\beta_{j_i}}(s_{j_i})$, then $u=\chi_{\beta_{j_1}}(s_{{j_1}})u''$, where $u''$ satisfies the inductive hypothesis. 

 Let $\nu$ be as given in Lemma~\ref{welong}. 
Take $\mu\in\wts(V)$ as defined in Equation \ref{nnu}.
Since $V_{\mu}=V_{(w_{\beta_{j_1}}\nu)}$,
by Lemma~\ref{ww},
the weight space $V_{\mu}$ is also one dimensional.



 We fix a basis element $v_{\mu}=x_{-\beta_{j_1}}^{(n')}\cdot v_{\nu}$  of $V_{\mu}$, where $n'$ is as given in Equation~\ref{nnu0}. Then $v_{\mu}\in V_{\Z}\cap V_{\mu}$ 
and hence by the assumption on $u$, we have
\begin{eqnarray}\label{peq1}
 u\cdot v_{\mu}\in V_{\Z} .  
\end{eqnarray}

We expand $u\cdot v_{\mu}=(\chi_{\beta_{j_1}}(s_{{j_1}})u'' x_{-\beta_{j_1}}^{(n')})\cdot v_{\nu}$ further and write it in its weight space decomposition
\begin{align*}
    (\chi_{\beta_{j_1}}(s_{{j_1}})u'' x_{-\beta_{j_1}}^{(n')})\cdot v_{\nu} &= x_{-\beta_{j_1}}^{(n')}\cdot v_{\nu}+(s_{{j_1}}x_{\beta_{j_1}}^{(1)}x_{-\beta_{j_1}}^{(n')})\cdot v_{\nu}+\dots
+(s_{{j_1}}^{n}x_{\beta_{j_1}}^{(n')}x_{-\beta_{j_1}}^{(n')})\cdot v_{\nu}\\
&\qquad\qquad+\text{weight vectors of lower depth than the depth of $\nu$}.
\end{align*}

Set $m=\langle \lambda, h_{\beta_{j_1}}\rangle$. Since $\lambda$ is regular, $\langle \lambda, h_{\beta_{j_i}}\rangle\neq 0$ for all $i=1,2,3,4,\dots n$. Then $m\le \langle \lambda, h_{\beta_{j_i}}\rangle$ for all $i=2,3,4,\dots n$, since the height of each $\beta_{j_i}$ is greater than or equal to the height of $\beta_{j_1}$. We have  
$$\{\chi_{\beta_{j_1}}(s_{{j_1}})u''\cdot x_{-\beta_{j_1}}^{(n')}v_{\nu}\}\cap V_{\nu}=\{(s_{{j_1}}^{n'}x_{\beta_{j_1}}^{(n')}x_{-\beta_{j_1}}^{(n')})\cdot v_{\nu}\}$$

Combining this with Equation (\ref{peq1}),  we have
$$(s_{{j_1}}^{n'}x_{\beta_{j_1}}^{(n')}x_{-\beta_{j_1}}^{(n')})\cdot v_{\nu}\in V_{\Z}\cap V_{\nu}.$$

By Lemma~\ref{lowest} $\mu$ is the lowest weight of an $\mathfrak{sl}_{\beta_{j_1}}$-module, $w_{\beta_{j_1}}\mu=\nu$ is the highest weight of this module. More precisely $\nu+\beta_{j_1}$ is not a weight of the $\mathfrak{sl}_{\beta_{j_1}}$-module and hence by Lemma~\ref{L-xx}, $s_{{j_1}}^{n'}\cdot v_{\nu}\in V_{\Z}$. This implies $s_{{j_1}}^{n'}\in \Z$ and thus $s_{{j_1}}\in \Z$. So, we get
$$u''\cdot V_{\Z}=\chi_{\beta_{1}}(-t_{\beta_{j_1}})u\cdot V_{\Z}\subseteq V_{\Z}$$
and hence by the inductive hypothesis $t_{\beta_{j_{i}}}\in \Z$, for all $i=2,3,\dots, n.$
This completes the induction and the proof of the statement follows.
\end{proof}

\section{ Toral subgroup}

 Recall from Equation (\ref{torele}) that for $\mu\in{\rm wts} (V)$, the elements $h_{i}(t)$ 
 act on weight vectors $v\in V_\mu$ as $$h_i(t)\cdot v = t^{\langle\mu,h_{\a_i}\rangle }v.$$

\begin{theorem}\label{toral} Let $V$  be a finite dimensional faithful $\frak g$-module. 
Let  $H(\Q)\leq G(\Q)$ be the subgroup generated by the elements $h_i(t)$,  $t\in \Q^{\times}$, $i\in I$. 
\begin{enumerate}
\item  The group $H(\Q)$ is abelian.
\item Every element of $H(\Q)$ may be expressed in the form $\prod_{i\in I} h_i(t_i)$, for $t_i\in \Q^{\times}$.
\end{enumerate}
\end{theorem}
\begin{proof} Let $\alpha,\beta\in\Delta$.  Let  $\mu\in {\rm wts}(V)$. For each weight $\mu$ of $V$, and $s,t\in \Q^{\times}$, 
$h_{\alpha}(s)h_{\beta}(t)$ acts on the weight space $V_{\mu}$  as $s^{\langle\mu,\a^{\vee}\rangle }t^{\langle\mu,\beta^{\vee}\rangle }$. Thus $h_{\alpha}(s)h_{\beta}(t)=h_{\beta}(t)h_{\alpha}(s)$. Since this is true for each representation $V$, we obtain
$h_{\alpha}(s)h_{\beta}(t)=h_{\beta}(t)h_{\alpha}(s)$. Similarly $h_{\a}(s)h_{\a}(t)=h_{\a}(st)$, which proves~(1).

For (2), let  $\a^{\vee}=\sum_{i\in I} c_i\a_i^{\vee}$. Then for $t\in \Q^{\times}$ and each weight $\mu$ of $V$,
$$t^{\langle\mu,\a^{\vee}\rangle }=t^{\langle\mu,\sum_{i\in I} c_i\a_i^{\vee}\rangle }=\prod_{i\in I}t^{c_i\langle\mu,\a_i^{\vee}\rangle }.$$
But $h_{\alpha}(t)$ acts as $t^{\langle\mu,\a^{\vee}\rangle }$ on $V^{\mu}$ and $\prod_{i\in I} h_i(t^{c_i})$ acts as $\prod_{i\in I}t^{c_i\langle\mu,\a_i^{\vee}\rangle }$ on $V^{\mu}$. Since this is true for every $V$, we obtain
$$h_{\alpha}(t)=\prod_{i\in I} h_i(t^{c_i}),$$
so every $h_{\alpha}(t)$ can be written in the desired form. Since $H(\Q)$ is generated by the $h_i(t)$ and is abelian, every element of $H(\Q)$ can be written in the desired form. 
\end{proof}

\begin{theorem}  
Let $V$ be a finite dimensional faithful $\frak g$-module  such that the set of weights of $V$ contains all the  fundamental weights. Then the map $t\mapsto h_i(t)$ is an injective homomorphism from $\Q^{\times}$ into $H(\Q)$, for each $i\in I$.
\end{theorem}
\begin{proof}
Let $h_i(t)\in H(\Q)$. Since each fundamental weight  $\omega_i$, $i\in I$,  is a weight of $V$, $h_i(t)$ acts as scalar multiplication on the weight space  $V_{\omega_i}$   by
$$t^{\langle\omega_i,\a_i^{\vee}\rangle }=t.$$ Hence for each $i\in I$, if  $h_i(t)=1$ then  $t=1$.
\end{proof}

\begin{corollary}\label{C-injective}   
Let $V$ be a $\frak g$-module   such that the set of weights of $V$ contains all the  fundamental weights. Then the map $$(\Q^{\times})^{\ell}\longrightarrow H(\Q)$$
$$(t_1,t_2,\dots ,t_{\ell})\mapsto\prod_{i\in I} h_i(t_i)$$ 
is an injective homomorphism.
\end{corollary}
\begin{proof}
For each $i\in I$, the restriction $t_i\mapsto h_i(t_i)$ of this map to the $i$-th factor of $(\Q^{\times})^{\ell}$ is injective. 
\end{proof}

 We note that the hypothesis of Corollary~\ref{C-injective}, that the set of weights contains all the fundamental weights, is satisfied when $G$ is a simply connected Chevalley group or whenever the root lattice $Q$ is isomorphic to the weight lattice $P$.
\section{Iwasawa decomposition of $G$}
  We recall the Bruhat decomposition of $G$, which holds over any field \cite{Steinberg}:
$$G=BWB =\bigsqcup_{w\in W}BwB$$ 
where $B=HU$ is
the Borel subgroup of $G$. We also write $B=B(\Q)$.

As before, let $V$ denote a finite dimensional faithful $\frak g$-module with highest weight $\lambda$. We assume that $L_{V}=P$. Let $G({\mathbb{Q}})$ be the corresponding simply connected Chevalley group.

\begin{lemma}\label{SL2Bruhat} We have $\SL_2({\mathbb{Q}})=\SL_2(\Z)B(\SL_2({\mathbb{Q}}))$, where $B(\SL_2({\mathbb{Q}}))$ denotes the Borel subgroup of $\SL_2({\mathbb{Q}})$.
\end{lemma}

\begin{proof}
Let $\left(\begin{matrix} a & b \\ c & d \end{matrix}\right)\in \SL_2(\Z)$ and let $\left(\begin{matrix} p & q \\ r & s \end{matrix}\right)\in \SL_2(\Q)$. Consider the product
$$\left(\begin{matrix} a & b \\ c & d \end{matrix}\right)\left(\begin{matrix} p & q \\ r & s \end{matrix}\right)\in \SL_2({\mathbb{Q}}).$$
We may choose $c,d\in\Z$ with $\gcd(c,d)=1$ and $cp+dr=0$. This determines $a,\ b\in\Z$ such that $ad-bc=1$. Then
$$\left(\begin{matrix} p & q \\ r & s \end{matrix}\right)=\left(\begin{matrix} a & b \\ c & d \end{matrix}\right)^{-1}\left(\begin{matrix} \a & \b \\ 0 & \gamma \end{matrix}\right)\in \SL_2(\Z)B(\SL_2(\Q)).\ \qedhere$$    
\end{proof} 

See also \cite{Steinberg}, Lemma 43(c), p.\ 101.

\begin{lemma}[\cite{Steinberg}, Lemmas 48 and 49]\label{integralimage}  For fixed $i$, $$\varphi_i(\SL_2(\Z))=\langle\chi_{\alpha_i}(s),\chi_{-\alpha_i}(t)\mid s,t\in\Z\rangle.$$
\end{lemma}

Now let $Y_i$ be a system of coset representatives for $Bw_iB/B$. A natural choice is \cite[Lemma 43(a)]{Steinberg}
$$Y_i=\{\chi_{\a_i}(t)w_i\mid t\in\Q\}.$$
\begin{lemma}\label{intcosetrep} The $Y_i$ can be chosen with integral coordinates. That is, we may choose coset representatives $Y_i=\{\chi_{\a_i}(t)w_i\}$  for $Bw_iB/B$ such that 
$Y_i\subseteq \varphi_i(\SL_2(\Z))$. \end{lemma}

\begin{proof}
In $\SL_2(\Q)$ we have 
$$\chi_{\a_i}(t)w_i=\left(\begin{matrix} 1 & t \\ 0 & 1 \end{matrix}\right)\left(\begin{matrix} 0 & 1 \\ -1 & 0 \end{matrix}\right)=
\left(\begin{matrix} -t & ~1 \\ -1 & ~0 \end{matrix}\right)\in\SL_2(\Q).$$
By Lemma~\ref{SL2Bruhat}, we can write an element of $\SL_2(\Q)$ in $\SL_2(\Z)B(\SL_2({\mathbb{Q}}))$:
$$\left(\begin{matrix} -t & ~1 \\ -1 & ~0 \end{matrix}\right)=\left(\begin{matrix} -t & ~t+1 \\ -1 & ~1 \end{matrix}\right)\cdot \left(\begin{matrix} 1 & 1 \\ 0 & 1 \end{matrix}\right).$$
Modulo $B$, the $Y_i$ can thus be chosen such that $t\in\Z$ and so  
$$Y_i=\varphi_i\left(\left(\begin{matrix} 1 & t \\ 0 & 1 \end{matrix}\right)\cdot \left(\begin{matrix} 0 & 1 \\ -1 & 0 \end{matrix}\right)\right)\subseteq \varphi_i(\SL_2(\Z)).  \qedhere$$
\end{proof}
It follows that for each $i$, 
$$Y_i\subseteq \Delta_i(\SL_2(\Z))=\langle\chi_{\alpha_i}(s),\chi_{-\alpha_i}(t)\mid s,t\in\Z\rangle\leq G(\Z).$$

\begin{lemma} For each $i\in I$,
$$Bw_iB= Y_iB,$$
where $Y_i$ is a system of coset representatives for $Bw_iB/B$.
\end{lemma}

See also  \cite{Steinberg}, Theorem 15, p 99.

\begin{corollary}\label{base} 
For each $i\in I$, we have $Bw_iB=Y_iB\leq G(\Z)B(\Q).$
\end{corollary} 

For example, an element $\left(\begin{matrix} 1/2 & ~0 \\ 3/4 & ~2 \end{matrix}\right)\in \SL_2(\Q)$ that is not in upper triangular form can be written in $\SL_2(\Z)B(\SL_2(\Q))$:
$$\left(\begin{matrix} 1/2 & ~0 \\ 3/4 & ~2 \end{matrix}\right)=\left(\begin{matrix} 2 & 1 \\ 3 & 2 \end{matrix}\right)
\left(\begin{matrix} 1/4 & -2 \\ 0 & ~4 \end{matrix}\right).$$

The following theorem gives an analog of Lemma~\ref{SL2Bruhat} for $G({\mathbb{Q}})$.
\begin{theorem}[\cite{Steinberg}, Theorem 18, p.\ 114]\label{Bruhat} We have
$G({\mathbb{Q}})=G(\Z)B({\mathbb{Q}})$.
\end{theorem}

  {\it Proof:} We prove Theorem~\ref{Bruhat} using the Bruhat decomposition. That is, we prove by induction on the length of $w\in W$ that each Bruhat cell $B({\mathbb{Q}})wB({\mathbb{Q}})$ is contained in $G(\Z)B({\mathbb{Q}})$. The case $\ell(w)=0$ gives 
$B({\mathbb{Q}})\leq G(\Z)B({\mathbb{Q}})$. The case $\ell(w)=1$ follows from Corollary~\ref{base}, which gives a base for induction.

Assume inductively that
$$B({\mathbb{Q}})w_{i_1}w_{i_2}\dots w_{i_k}B({\mathbb{Q}})\leq  G(\Z)B({\mathbb{Q}}),$$
for all $k\geq 1$, where $w_{i_1}w_{i_2}\dots w_{i_k}$ has minimal length, and each $w_{i_j}$ is a simple root reflection. 

   Let  $w_{i_{k+1}}$ be a simple root reflection and assume that $w_{i_1}w_{i_2}\dots w_{i_k}w_{i_{k+1}}$ has minimal length. Then the rule B3 gives
$$B({\mathbb{Q}})w_{i_1}w_{i_2}\dots w_{i_k}w_{i_{k+1}}B({\mathbb{Q}})=B({\mathbb{Q}})w_{i_1}w_{i_2}\dots w_{i_k}B({\mathbb{Q}})w_{i_{k+1}}B({\mathbb{Q}}).$$
By the inductive hypothesis, $B({\mathbb{Q}})w_{i_1}w_{i_2}\dots w_{i_k}B({\mathbb{Q}})\leq  G(\Z)B({\mathbb{Q}})$. Since $w_{i_{k+1}}$  is a simple root reflection, the inductive hypothesis also implies that $$B({\mathbb{Q}})w_{i_{k+1}}B({\mathbb{Q}})\leq  G(\Z)B({\mathbb{Q}}).$$ 
Then $$B({\mathbb{Q}})w_{i_1}w_{i_2}\dots w_{i_k}B({\mathbb{Q}})w_{i_{k+1}}B({\mathbb{Q}})$$ 
is the product in $G(\Z)B({\mathbb{Q}})$ of the images of $B({\mathbb{Q}})w_{i_1}w_{i_2}\dots w_{i_k}B({\mathbb{Q}})$ and $B({\mathbb{Q}})w_{i_{k+1}}B({\mathbb{Q}})$. Thus this product also lies in $G(\Z)B({\mathbb{Q}})$. Hence each  each Bruhat cell $B({\mathbb{Q}})wB({\mathbb{Q}})$ is contained in $G(\Z)B({\mathbb{Q}})$.
$\square$


\section{Integrality of $G(\mathbb{Q})$}
Recall that $\Gamma(\Z) = \{g \in G(\Q)\mid g\cdot V_{\Z} = V_{\Z}\}$
is the subgroup of $G(\Q)$ that preserves $V_{\Z}$. 

\begin{theorem}\label{stabilizer1} In  $G({\mathbb{Q}})$ we have $G(\Z)\cap U({\mathbb{Q}})=\Gamma(\Z)\cap U({\mathbb{Q}})$. 
Thus any element of $U({\mathbb{Q}})$ that stabilizes $V_{\Z}$ lies in $G(\Z)$.
\end{theorem}

  {\it Proof:} We have $G(\Z)\cap U({\mathbb{Q}})\subseteq\Gamma(\Z)\cap U({\mathbb{Q}})$, since by Lemma~\ref{stabVZ}, $G(\Z)\subseteq\Gamma(\Z)$. For the reverse inclusion, let $\gamma\in \Gamma(\Z)\cap U({\mathbb{Q}})$. So $\gamma\in U({\mathbb{Q}})$ and $\gamma \cdot V_{\Z}=V_{\Z}$. Since $\gamma\in U({\mathbb{Q}})$, we have $\gamma=\prod_{\alpha\in\Delta^+} \chi_{\alpha}(t_{\alpha})$.  By Theorem~\ref{stabilizer0}, all $t_{\alpha}$ are integers. Hence $\gamma\in G(\Z)$. $\square$

\begin{theorem}\label{stabilizer2} Suppose that the set ${\rm wts}(V)$ of weights of $V$ contains all the fundamental weights. In  $G({\mathbb{Q}})$ we have 
$$G(\Z)\cap H({\mathbb{Q}})=\Gamma(\Z)\cap H({\mathbb{Q}})=H({\mathbb{Z}})=\langle h_i(-1)\mid i\in I\rangle.$$ 
\end{theorem}

  \begin{proof}
    It is immediate that $G(\Z)\cap H({\mathbb{Q}})=\langle h_i(-1)\mid i\in I\rangle$. The generator $h_i(t_i)$ acts on the weight space $V_{\omega_i}$ as scalar multiplication by $t_i$. Hence $h_i(t_i)$ preserves $V_{\Z}$ when $t\in\Z^{\times}$. Thus $\Gamma(\Z)\cap H({\mathbb{Q}})=\langle h_i(-1)\mid i\in I\rangle$.   
  \end{proof}

See also \cite{Steinberg}, Lemma 49(c), p.\ 114.

\begin{corollary}\label{stabilizer3} Suppose that the set ${\rm wts}(V)$ of weights of $V$ contains all the fundamental weights. In  $G({\mathbb{Q}})$ we have $G(\Z)\cap B({\mathbb{Q}})=\Gamma(\Z)\cap B({\mathbb{Q}})$. 
Thus any element of $B({\mathbb{Q}})$ that stabilizes $V_{\Z}$ lies in $G(\Z)$.
\end{corollary}

  This allows us to prove our main result.

\begin{theorem} Suppose that the set ${\rm wts}(V)$ of weights of $V$ contains all the fundamental weights. We have $\Gamma(\Z)=G({\mathbb{Z}})$. 
That is,  any element of $G({\mathbb{Q}})$ that stabilizes $V_{\Z}$ lies in the group~$G(\Z)=\langle\chi_{\a_i}(s),\chi_{-\a_i}(t)\mid s,t\in\Z,\ i\in I\rangle$.
\end{theorem}
\begin{proof}
We have  $G(\Z)\subseteq G({\mathbb{Q}})=G(\Z)B({\mathbb{Q}})$ by Theorem~\ref{Bruhat}. Let $\gamma\in G({\mathbb{Z}})$. Then  $$\gamma=\gamma_0b\in G({\mathbb{Q}})=G(\Z)B({\mathbb{Q}})$$ with $\gamma_0\in G(\Z)$ and $b\in B({\mathbb{Q}})$. Thus $\gamma_0^{-1}\gamma=b\in B$ and by Lemma~\ref{stabVZ}, $G(\Z)\subseteq\Gamma(\Z)$, so $$\gamma_0^{-1}\gamma\in \Gamma(\Z)\cap B({\mathbb{Q}})=G(\Z)\cap B({\mathbb{Q}})$$ by Corollary~\ref{stabilizer3}. Thus $\gamma_0^{-1}\gamma\in G(\Z)$, so $\gamma\in G(\Z)$.  It follows that $\gamma_0\in G({\mathbb{Z}})$.
\end{proof}  




\section{Generating sets for $G({\mathbb{Q}})$ and $G({\mathbb{Z}})$}
Our simply connected Chevalley group $G({\mathbb{Q}})$ is generated by $$\{ \chi_\a(u),\ h_i(t) \mid \a\in\Delta, u\in\Q,\  t\in \Q^\times,\ i\in I\}.$$
 This generating set is redundant and we may reduce it using Steinberg's group relations \cite{Steinberg}.

\begin{lemma}\label{reduction} Let $G({\mathbb{Q}})$ be a simply connected Chevalley group. Then 
\begin{align*}
    G({\mathbb{Q}})&=\langle\chi_{\alpha_i}(s),\ \widetilde{w}_{\alpha_i}(t),\ h_i(t)\mid s\in\Q,\ t\in\Q^{\times}, \ i\in I\rangle\\
&=\langle\chi_{\alpha_i}(s),\ \chi_{-\alpha_i}(u),\ h_i(t)\mid s,u\in\Q,\ t\in \Q^\times,\ i\in I\rangle.
\end{align*}
\end{lemma}
\begin{proof}
For all $\alpha\in \Delta$ there exists $w\in W$ such that $\a=w\a_i$ for some $i\in I$. We have 
$$\chi_{\a}(t)=\chi_{w\alpha_i}=\widetilde{w}  \chi_{\alpha_i}(\pm t)\widetilde{w}^{-1}$$
for a lifting $\widetilde{w}\in \widetilde{W}$ of $w\in W$ under the natural projection $\widetilde{W}\longrightarrow W$. The $\widetilde{w}_{\alpha_i}$ generate $\widetilde{W}$, thus the first equality follows. The second inequality is obvious since each $\widetilde{w}_{\alpha_i}$ is a product of $\chi_{\alpha_i}$'s and $\chi_{-\alpha_i}$'s.
\end{proof}

\begin{corollary} 
 The group $G(\mathbb{Z})$ has the following  generating sets:
\begin{enumerate}
    \item $\{\chi_{\alpha_i}(1),\ \chi_{-\alpha_i}(1),\ h_i(-1)\mid i\in I\}$, or
    \item $\{\chi_{\alpha_i}(1),\ \widetilde{w}_{\alpha_i}, \ h_i(-1)\mid i\in I\}$.
\end{enumerate}
\end{corollary}
\begin{proof}
    (1) is an immediate consequence of the exponential rule $$\chi_{\alpha_i}(s)=\chi_{\alpha_i}(1+1+\dots +1)=\chi_{\alpha_i}(1)^s$$ for $s\in\Z$. (2) follows immediately from (1).
\end{proof}

 In some cases (such as $G(\Z)=\SL_2(\Z)$), the elements $h_i(-1)$, for $ i\in I$, are not required as explicit generators of $G(\Z)$. 

\section{Integrality of inversion subgroups of Kac--Moody groups}\label{inversion}

For the remaining two sections, we assume that $\fg$ is a symmetrizable Kac--Moody algebra over $\C$ and that $(V,\rho)$ is an integrable $\fg$-module. That is, $V$ is a weight module and the Chevalley--Serre generators of $\fg$, denoted $e_i$ and $f_i$ for $i\in I$, act locally nilpotently on any $x\in V$. That is, for   each $x\in V$, there exists a natural number $n=n(i, x)$ such that
$\rho (e_i)^n\cdot x=\rho (f_i)^n\cdot x=0.$

In order to extend our representation-theoretic  results to a proof of integrality for Kac--Moody groups, we considered an alternate approach in \cite{ACLM}. We proved  integrality of inversion subgroups of $U(\Q)$, which are defined as follows.  
For $w$ in the Weyl group $W$, the {\it inversion subgroup} $U_{(w)}(\Q)\leq U(\Q)$ is 
 $$U_{(w)}(\Q)=\langle U_{\beta}(\Q)\mid \beta\in \Phi_{(w)}\rangle,$$ 
 where
$$
 \Phi_{(w)}=\{\beta\in \Delta^{\re}_{+}\mid w^{-1}\beta\in \Delta_{-}\}$$ 
 and $U_\beta(\Q)=\{ \chi_\beta(t)\mid t\in\Q\}$ where
$\chi_{\beta}(t)=\widetilde{w}\chi_{\alpha_i}(\pm t)\widetilde{w}^{-1}$ and $\beta=w\a_i$ for some $w\in W$. If  $w\in W$ has the reduced expression $w_{i_1}w_{i_2}\cdots w_{i_k}$ then $\Phi_{(w)}$ has cardinality $k=\ell(w)$. We set
$$U_{(w)}(\Z) = U_{(w)}(\Q)\cap U(\Z).$$

\begin{theorem}[\cite{ACLM}]\label{KMint} Let $U(\Q)$ be the positive unipotent subgroup of $G(\Q)$. For any $w\in W$,
$U_{(w)}(\Q)$ is  integral. That is for all $g\in U_{(w)}(\Q)$, $g\cdot V_\Z \subseteq V_\Z$ implies that $g\in U_{(w)}(\Z)$.\end{theorem}

Theorem~\ref{KMint} holds in the case that $\mathfrak g$ is a semisimple Lie algebra and  $V$ is any highest weight $\mathfrak g$-module.  In this case, the Weyl group $W$ is finite and $W$ contains the longest element, denoted $w_{0}$, which flips all positive roots to negative roots. That is $\Phi_{(w_{0})}=\Delta^{+}$, which gives $U_{(w_{0})}=U(\Q)$.
It follows that $G(\Z)\cap U(\Q)=U(\Z)$.

Unfortunately, Kac--Moody algebras have infinite Weyl groups and there is no analog of the longest element which flips all positive roots to negative roots. Hence this method cannot be used to prove integrality of unipotent subgroups in Kac--Moody groups.

\section{Integrality of Kac--Moody groups  over $\Q$}\label{KMcase}

Part of the motivation for this work is to extend the proof of integrality of Chevalley groups given here to integrality of representation theoretic Kac--Moody groups $G_V(\Q)$ over $\Q$. The group $G_V(\Q)$ is constructed  with respect to an integrable highest weight module $V=V^{\lambda}$ for the underlying symmetrizable Kac--Moody algebra $\mathfrak{g}$, with dominant integral highest weight~$\lambda$.

Over a general commutative ring $R$, there is still no widely agreed upon definition of a Kac--Moody group $G(R)$, except in the affine case (\cite{Allcock}, \cite{Garland}). It is expected be some sort of generalization of the notion of a Chevalley--Demazure group scheme. The groundwork for the notion of a Kac--Moody group over an arbitrary commutative ring was provided by Tits who defined a functor $\widetilde{\mathfrak{G}}$ from commutative rings to groups \cite{Ti87}. 

There are two natural candidates for a representation theoretic Kac--Moody group over $\Z$:
\begin{itemize}
\item A representation theoretic construction $G_V(\Z)$ of the value of the  Tits functor $\widetilde{\mathfrak{G}}$  over $\Z$. 
\item The subgroup $\Gamma_V(\Z)=\{g\in G_V(\Q)\mid g\cdot V_{\Z}=V_\Z\}$  preserving an integral lattice $V_{\Z}$ in~$V$.
\end{itemize}

A proof of integrality for $G_V(\Q)$ would therefore give a unique definition of representation theoretic Kac--Moody groups over $\Z$.

There are certain difficulties and open questions  in answering this question for Kac--Moody groups{\footnote{An earlier preprint by a subset of the authors (arXiv:1803.11204v2 [math.RT] Proposition 6.4) contains an error.}}. For example:

\begin{enumerate}
    \item It is not known how to find a $\Z$-basis of $\frak{g}_{\a, \Z}$ for every $\a\in \Delta^+$, such that all structure constants with respect to this basis are integral.
    \item It is not known how to find a $\Z$-basis for $V_\mu^\lambda$ for every weight $\mu$ of $V^\lambda$ such that all structure constants with respect to this basis are integral.
    \item It is not known if the inclusion 
    $\mathfrak{g}_{\a,\Z}\otimes_\Z \C=
    (\mathfrak{g}_{\a}\cap \mathcal{U}_{\Z})\otimes_\Z \C \subseteq \frak g_\a$ is an equality. That this, if  $\mathfrak{g}_{\a,\Z}$ spans $\mathfrak{g}_{\a}$. 
\end{enumerate}

A solution to (1) would give a solution to (2) and 
a solution to (3) would give a solution to (1).
An answer to these questions would allow us to prove an analog of Theorem~\ref{stabilizer1} for Kac--Moody groups. We hope that we will be able to prove integrality of Kac--Moody groups over $\Q$ in future work.

\bibliographystyle{amsalpha}
\bibliography{GPmath}{}

\end{document}